%% file: main.tex
\documentclass[10pt]{article}

\input{preamble}

\title{Gradient descent reaches arbitrary accuracy in a fixed number of large steps on multiclass logistic regression}
\author{
  Andreea-Alexandra Musat and Nicolas Boumal\thanks{\'Ecole Polytechnique F\'ed\'erale de Lausanne (EPFL), Institute of Mathematics,
      \texttt{\{andreea.musat, nicolas.boumal\}@epfl.ch}.}
}
\date{\today}

\begin{document}

\maketitle

\begin{abstract}
    We show that gradient descent (GD) with step size $\eta$ can achieve $\varepsilon$ optimality gap on the multiclass logistic regression loss in a number of iterations $T$ that is \emph{independent} of $\varepsilon$.
    To do so, the step size $\eta$ must be chosen large enough as a function of $\varepsilon$ (and properties of the linearly separable dataset), while $T$ is a constant that only depends on the data.
    The relevant properties of the data include the number of points and their margin (for $T$), and one additional quantity (for $\eta$) that turns out to be necessary in dimension $d \geq 3$.
    The proofs build on the known fact that, as $\eta \to \infty$, suitably rescaled GD becomes the perceptron algorithm, and the latter finds a perfect classifier in $T$ iterations.
    A similar claim was shown for two classes in dimension $d = 2$ using other proof techniques.

    This result adds to previous works that showed GD on the logistic regression loss can be accelerated by taking large step sizes.
    Those provoke a wild, oscillatory initial phase, followed by a stable phase.
    We provide results of this (qualitatively different) nature, too.
\end{abstract}

\section{Introduction}

Traditional convex optimization theory prescribes a ``small'' step size for gradient descent.
For an $L$-smooth convex objective, a step size $\eta < 2/L$ ensures that the loss decreases at every iteration and yields the standard $\mathcal O(1/T)$ convergence rate in function value.

For logistic regression on linearly separable data, \citet{wu2024large} studied gradient descent (GD) with constant step sizes well above this classical threshold.
Although the loss may initially oscillate and the usual descent lemma no longer applies, they showed that GD still converges and, perhaps surprisingly, the loss decays at an $\widetilde{\mathcal O}(1/T^2)$ rate, improving on the standard bound.
In this setting, the nonmonotonic behavior usually associated with instability can instead improve optimization efficiency.

The analysis of \citet{wu2024large} identifies two phases in the dynamics.
The iterates first pass through an \emph{unstable region}, in which the loss may oscillate, before entering a \emph{stable region} after $\tau = \mathcal O(\eta)$ iterations.
Once in the stable region, the loss decreases at every iteration.
Then, for an iteration budget $T$, they show that choosing $\eta = \Theta(T)$ yields an accelerated $\widetilde{\mathcal O}(1/T^2)$ bound on the final loss.
However, since their bound on the transition time $\tau$ grows linearly with $\eta$, their result suggests that a larger step size may prolong the unstable phase.

In contrast, for logistic regression \emph{in two dimensions}, \citet{crawshaw2026tight} recently showed that, for sufficiently large step sizes, the number of iterations $\tau$ required to enter the stable region is bounded independently of $\eta$.
In addition, they showed that any target loss value $\varepsilon$ can be reached within a fixed number of iterations, independent of $\varepsilon$, by taking $\eta$ sufficiently large.
Thus, higher accuracy is obtained by increasing the step size rather than the number of iterations.

We show that these phenomena extend to multiclass logistic regression \emph{in all dimensions}.
Specifically, let $K \geq 2$ be the number of classes and consider a dataset $\{ (x_i, y_i) \}_{i=1}^n$, where $x_i \in \reals^d$ and $y_i \in \{ 1, \dots, K \}$.
We consider a linear softmax classifier parameterized by $W \in \reals^{K \times d}$, whose $k$th row is denoted by $w_k \in \reals^d$.
The associated \emph{multiclass cross-entropy loss} is
\begin{align} \label{eq:multiclass-loss}
    f(W) \coloneqq
    -\frac{1}{n} \sum_{i=1}^n \sum_{k=1}^K Y_{ik} \log\bigl( P(W)_{i k} \bigr),
    \qquad
    P(W)_{ik} \coloneqq
    \frac{\exp\bigl(\langle w_k,x_i\rangle\bigr)}{\sum_{j=1}^K \exp\bigl(\langle w_j,x_i\rangle\bigr)}.
\end{align}
Here, $\langle w_k,x_i\rangle$ represents the \emph{score} of class $k$ for data point $x_i$, and $P(W) \in \reals^{n \times K}$ is the matrix of row-wise softmax probabilities.
Lastly, $Y \in \reals^{n \times K}$ is the one-hot encoding of the labels, that is,
\begin{align} \label{eq:one-hot-labels}
    Y_{iy_i}=1,
    \qquad \textrm{and} \qquad
    Y_{ik}=0 \text{ for all } k\neq y_i.
\end{align}
We assume that the dataset is linearly separable and consider full-batch GD on $f$ with constant step size $\eta > 0$, whose sequence of iterates $\{\WWeta{t}\}_{t\geq0}$ is generated by the update rule
\begin{align*}
    \WWeta{t+1} = \WWeta{t} - \eta \nabla f(\WWeta{t}),
\end{align*}
initialized at $\WWeta{0}=0$.

\subsection{Main results}

We establish two complementary convergence guarantees for multiclass GD with large step sizes.
The first gives a two-phase convergence result in which the transition time grows at most linearly with the step size.
The second shows that, once the step size is larger than a data-dependent threshold, this dependence disappears and the transition occurs within a fixed number of iterations.

We now introduce the assumptions and notation needed to state these results precisely.
Throughout, we define $[n] = \{1, \dots, n\}$ and we use $\langle \cdot, \cdot \rangle$ and $\langle \cdot, \cdot \rangle_\mathrm{F}$ for the standard Euclidean inner products on $\reals^d$ and $\reals^{K \times d}$, respectively, and $\| \cdot \|$ and $\| \cdot \|_\mathrm{F} $ for the associated norms. 
We assume that the data satisfy the following:
\begin{enumerate}[label=\textbf{(A\arabic*)}]
    \item \label{ass:multiclass1} $\| x_i \| \leq 1$ for all $i \in [n]$.
    \item \label{ass:multiclass2} There exist $W^* \in \reals^{K \times d}$ and $\gamma > 0$ such that $\| W^* \|_\mathrm{F} \leq 1$ and
    \begin{align*}
        \langle w_{y_i}^*, x_i \rangle
        \geq
        \langle w_k^*, x_i \rangle + \gamma
        \qquad
        \text{ for all } i \in [n]
        \text{ and all } k \neq y_i.
    \end{align*}
\end{enumerate}
The assumptions are standard:~\ref{ass:multiclass2} requires the data to be linearly separable with positive margin, while the bounds $\| x_i \| \leq 1$ and $\|W^*\|_\mathrm{F} \leq 1$ merely fix the scale of the problem and are without loss of generality.
Note that the margin $\gamma$ need not be maximal: defining
\begin{align} \label{eq:multiclass-max-margin}
    \gamma_{\max}
    \coloneqq
    \max_{\|W\|_\mathrm{F} \leq 1}
    \min_{i \in [n]}
    \min_{k \neq y_i}
    \langle w_{y_i}-w_k,x_i\rangle,
\end{align}
our results hold for any admissible $\gamma$, and in particular with $\gamma=\gamma_{\mathrm{max}}$.

\begin{remark}
    For $K=2$, the multiclass problem reduces to binary logistic regression under the identification $w = w_1 - w_2$, with $f(W) = f_{\mathrm{bin}}(w)$, $\eta_{\mathrm{bin}} = 2 \eta$, and $\gamma_{\mathrm{bin}} = \gamma_\mathrm{max} / \sqrt{2}$ (see Appendix~\ref{appendix:K2-binary-equivalence}).
\end{remark}

For GD with step size $\eta$, we define the \emph{stable region} as the set where $f(W) < 1/(2 \eta)$, a threshold below which we obtain monotonic descent at every subsequent iteration (Lemma~\ref{lemma:descent-lemma-if-stable}).
To quantify when the iterates enter this phase, we introduce the \emph{transition time}
\begin{align} \label{eq:tau-def}
    \tau
    \coloneqq
    \inf \left\{ t \geq 0 \, \middle| \, f(\WWeta{t}) < \frac{1}{2\eta} \right\},
\end{align}
with the convention that $\inf\varnothing = +\infty$.

We now state informal versions of our main results, suppressing constants and technical conditions on the step size.
The first result extends the two-phase convergence analysis of GD on the logistic regression loss by~\citet{wu2024large} from the binary to the multiclass setting.
The proof, however, is different: it relies on a multiclass perceptron inequality (Lemma~\ref{lemma:main-perceptron}) and yields a $\mathcal O(1/T^2)$ last-iterate rate, avoiding the additional logarithmic factors in \citep[Cor.~2]{wu2024large}.
The corresponding quantitative statements are given in Theorem~\ref{thm:multiclass-two-phase} and Corollary~\ref{cor:multiclass-accelerated-rate}.

\begin{theorem} (informal) \label{thm:multiclass-two-phase-informal}
    Run GD with constant step size $\eta \geq \frac{n}{2}$ on $f$~\eqref{eq:multiclass-loss}, initialized at $\WWeta{0} = 0$.
    As defined, the iterates enter the stable region $\{ W \mid f(W) < 1/2\eta \}$ after $\tau$~\eqref{eq:tau-def} iterations.
    Then:
    \begin{enumerate}[label=(\roman*)]
    \item (\textbf{Transition time.})
        The transition time $\tau$ satisfies $\tau \leq \mathcal O(\frac{\eta + \log(\eta K)}{\gamma^2})$.

    \item (\textbf{Loss decrease in the stable region.})
        For all $t > \tau$, the iterates remain in the stable region and the loss satisfies
        $f(\WWeta{t}) = \mathcal O \left( \frac{1}{\eta \gamma^2 (t-\tau)} \right)$.
    \end{enumerate}
\end{theorem}
    
Theorem~\ref{thm:multiclass-two-phase-informal} highlights a trade-off: increasing the step size may linearly increase the time required to enter the stable region, but reduces the loss bound once it is there.
Balancing these two effects by choosing $\eta$ in terms of the iteration budget $T$ gives the following accelerated rate for the last iterate.

\begin{corollary} \label{cor:multiclass-accelerated-rate-informal}
    Assume that the iteration budget satisfies $ T \gtrsim \frac{n+\log K}{\gamma^2} $, and run GD with a suitable constant step size $ \eta = \Theta( \gamma^2 T) $.
    Then, the iterates enter the stable region within the first half of the iteration budget, that is, $\tau \leq \frac{T}{2}$, and the loss at iteration $T$ satisfies $f(\WWeta{T}) = \mathcal O \left(\frac{1}{\gamma^4 T^2} \right)$.
\end{corollary}

Theorem~\ref{thm:multiclass-two-phase-informal} shows that the transition time grows at most linearly with $\eta$.
In fact, once the step size is sufficiently large, it becomes independent of $\eta$ (even though the stable region keeps shrinking).
For similar reasons (and surprisingly to us), by increasing the step size, we can make the loss arbitrarily small in a number of iterations that is independent of the target loss.

\begin{theorem} (informal) \label{thm:multiclass-main-up-informal}
    Fix a dataset $\{ (x_i, y_i) \}_{i = 1}^n$ satisfying~\ref{ass:multiclass1} and~\ref{ass:multiclass2}. 
    There exist a constant $\delta > 0$ and an iteration count $T \leq \left\lfloor 4n / \gamma^2 \right\rfloor + 1$ (that depend only on this dataset)
    such that, if we optimize $f$ in~\eqref{eq:multiclass-loss} using GD initialized at $\WWeta{0} = 0$ with any constant step size $\eta$, then 
    \begin{enumerate}[label=(\roman*)]
    \item (\textbf{Transition time.})
        There exists a threshold $\eta_{\mathrm{tr}}^* = \eta_{\mathrm{tr}}^*(\delta,T,K)$ such that, if $\eta \geq \eta_{\mathrm{tr}}^*$, the iterates enter the stable region in $\tau \leq T$ iterations.
    \item (\textbf{$\varepsilon$-loss at a fixed iteration $T$ independent of $\eta$ and $\varepsilon$.})
        For any $\varepsilon \in (0,1)$, there exists a threshold $\eta_{\varepsilon}^* = \eta_{\varepsilon}^*(\delta,T,K,\varepsilon)$ such that, if $\eta \geq \eta_{\varepsilon}^*$, then $f(\WWeta{T}) < \varepsilon$.
    \end{enumerate}
\end{theorem}

For binary logistic regression in two dimensions $(K=2, d=2)$, \citet{crawshaw2026tight} showed a similar result, but with uniform step size thresholds depending only on the coarser dataset parameters $n$ and $\gamma$.
Specifically, they showed that there exist $\eta_0(n, \gamma)$ and $T(n, \gamma)$ such that GD with any step size $\eta \geq \eta_0(n, \gamma)$ on any dataset with $n$ samples and margin $\gamma$ enters the stable region after $\tau \leq T(n, \gamma)$ iterations.
In Theorem~\ref{thm:counterexample}, we further show that such a uniform result does not extend to higher dimensions.
This justifies the appearance of a new data-dependent constant such as $\delta$.

Unlike \citet{crawshaw2026tight}, whose analysis relies on the region outside the stable one being disconnected (which holds only for $d = 2$), our proof relies on taking $\eta \to \infty$.
As already identified by~\citet{tyurin2025logistic} in the binary case, taking $\eta \to \infty$ produces a well-defined limiting dynamical system, namely, a batched multiclass hard perceptron~\eqref{eq:multiclass-infinity-system}.
Using the perceptron inequality (Lemma~\ref{lemma:main-perceptron}), we show that the~\eqref{eq:multiclass-infinity-system} dynamics produce a perfect classifier after $T \leq \left\lfloor 4n / \gamma^2 \right\rfloor + 1$ steps, and subsequently terminates.
Then, for every fixed dataset, we quantitatively show that once the step size exceeds a data-dependent threshold, the rescaled GD iterates closely track the corresponding perceptron iterates up to time $T$.
In particular, this implies that the $T$-th iterate of GD is also a perfect classifier, which ultimately implies a bound on $f(\WWeta{T})$ that decays exponentially in $\eta$ (Proposition~\ref{prop:loss-decrease-with-shadowing}).
This is the key ingredient for both statements of Theorem~\ref{thm:multiclass-main-up-informal}.

\begin{remark}
    Theorems~\ref{thm:multiclass-two-phase-informal} and~\ref{thm:multiclass-main-up-informal} describe different large step size regimes.
    Corollary~\ref{cor:multiclass-accelerated-rate-informal} gives an explicit $\mathcal O(1/T^2)$ guarantee by choosing the step size $\eta$ as a function of the iteration budget $T$, whereas Theorem~\ref{thm:multiclass-main-up-informal} shows that, beyond a data-dependent threshold, the transition time itself becomes independent of the step size.
    Thus, neither result subsumes the other.
\end{remark}

\subsection{Proof ideas}

To analyze gradient descent with large step sizes, we draw on its close resemblance to the perceptron algorithm \citep{rosenblatt1958perceptron,novikoff1962convergence}.
In Section~\ref{section:general-system}, we introduce a \emph{general batched perceptron dynamics}~\eqref{eq:general-system}.
Using its update rule $\ZZeta{t+1} = \ZZeta{t} + \Delta^{(t)}$ (with $\Delta^{(t)}$ specified in Section~\ref{section:general-system}), we bound the progress of the alignment $\langle \ZZeta{t}, W^* \rangle_\mathrm{F}$ from below (Lemma~\ref{lemma:good-direction-progress}), and the iterate norm $\| \ZZeta{t} \|_\mathrm{F}$ from above (Lemma~\ref{lemma:iterate-squared-norm}).
Both estimates depend on the cumulative \emph{average mistake mass} up to time $t$ (defined in Eq.~\eqref{eq:avg-mistake-def}).
Combining these bounds with the Cauchy--Schwarz inequality gives the main perceptron inequality (Lemma~\ref{lemma:main-perceptron}), which we use for both of our main results.

To study the GD trajectory for large step sizes $\eta$ and ultimately take the limit $\eta \to \infty$, we must normalize the iterates.
We therefore introduce the rescaled variables $\ZZeta{t} = \WWeta{t} / \eta$.
After this change of variables, GD takes the form of the \emph{soft perceptron} dynamics~\eqref{eq:multiclass-eta-system}, an instance of the general system~\eqref{eq:general-system}.
Specializing the main perceptron inequality (Lemma~\ref{lemma:main-perceptron}) to this soft perceptron yields the refined inequality of Lemma~\ref{lemma:finite-eta-perceptron-inequality}, which bounds the cumulative average mistake mass along the trajectory.

To turn this inequality into a bound on the transition time $\tau$, we adapt the standard perceptron argument (Section~\ref{subsection:hard-perceptron-limit}).
Instead of evaluating the perceptron inequality at the first perfect classifier iterate, we evaluate it at the first time $\tau$ at which the dynamics enters the stable region as in~\eqref{eq:tau-def}.
The definition of $\tau$ then gives a lower bound on the mistake mass accumulated before time $\tau$, which, combined with the refined perceptron inequality, yields $\tau=\mathcal O(\eta)$.
The remainder of the analysis then follows standard arguments \citep{wu2024large,crawshaw2026tight}: once the dynamics enters the stable region, the loss decreases at every subsequent iteration.

For our second result, we study the rescaled dynamics in the limit $\eta\to\infty$.
Crucially, this limiting iteration map is well defined: the row-wise softmax becomes a row-wise \emph{hardmax} that splits mass uniformly among classes with maximal score (Lemma~\ref{lemma:eta-infinity-limit}).
The limiting dynamics is a multiclass \emph{hard perceptron dynamics}~\eqref{eq:multiclass-infinity-system}, and the standard argument above gives a finite termination time $T = \mathcal O \left( n / \gamma^2 \right)$ at which it reaches a perfect classifier.

To transfer this conclusion to GD, we show that the soft perceptron~\eqref{eq:multiclass-eta-system} \emph{shadows} the hard perceptron~\eqref{eq:multiclass-infinity-system} trajectory up to time $T$ if the step size is large enough.
The required step size depends on the smallest positive gap $\delta$ between a maximal and a nonmaximal score encountered along the~\eqref{eq:multiclass-infinity-system} trajectory (defined in Eq.~\eqref{eq:delta-def}).
This gap is positive because only finitely many iterates and class scores are involved.
An induction on $t$ shows that, up to time $T$, the~\eqref{eq:multiclass-eta-system} scores remain close to the~\eqref{eq:multiclass-infinity-system} ones, which in particular implies that GD also produces a perfect classifier at time $T$.
Additionally, the same argument yields a bound on the loss that decays exponentially in $\eta$, giving both the bound on the transition time $\tau \leq T$ and, after increasing $\eta$ if necessary, $f(\WWeta{T}) < \varepsilon$ for any prescribed $\varepsilon > 0$.

\section{Related work}

\paragraph{Optimization with large step sizes.}
Classical analyses of gradient descent typically impose a condition on the step size that guarantees a decrease in the objective at every iteration.
Nevertheless, empirical studies have shown that GD can still converge despite this nonmonotonic behavior~\citep{smith2019super,lewkowycz2020large,cohen2021gradient}.
More recent theoretical work has established convergence in several such unstable regions and has proposed mechanisms including progressive sharpening, self-stabilization, and the interaction between successive gradient descent steps~\citep{ahn2022understanding,arora2022understanding,damian2022self,chen2023beyond}.

A distinct line of work considers non-constant step size schedules that contain occasional long steps.
Although such steps need not decrease the objective individually, their cumulative effect can accelerate gradient descent on smooth convex problems.
Examples include fractal step size schedules~\citep{agarwal2021acceleration}, long-step schedules~\citep{grimmer2024provably,grimmer2025accelerated}, the silver step size schedule of~\citet{altschuler2025acceleration} and schedules involving restarts~\citep{clement2026gradient}.

\paragraph{Binary logistic regression with large step sizes.}
For linearly separable logistic regression,~\citet{wu2023implicit} proved that, for any constant step size, the loss converges to zero and the iterates converge in direction to the maximum-margin separator.
Subsequently,~\citet{wu2024large} showed that large step sizes can improve the convergence rate: after an initial nonmonotone phase of at most $\mathcal O( \eta )$ iterations, gradient descent enters a region in which the loss decreases at each step.
In two dimensions,~\citet{crawshaw2026tight} showed that, for any $\eta$ large enough, the length of this initial unstable phase is in fact independent of $\eta$.
Concurrent works have further refined this picture: \citet{gong2026improved} showed that the transition time is upper bounded by an expression polylogarithmic in $\eta$, while \citet{wen2026tight} showed that a similar polylogarithmic dependence in $\eta$ is tight in the worst case.

The connection with perceptron dynamics was studied by~\citet{tyurin2025logistic}, who showed that, for binary logistic regression, rescaled GD converges to a batch perceptron as the step size tends to infinity, and analyzed the resulting iteration complexity for finding a separating classifier.
Extensions to this phenomenon include regularized logistic regression~\citep{wu2026large}, local gradient descent on heterogeneous data~\citep{crawshaw2025constant}, and more general Fenchel--Young losses~\citep{bao2026any}.
For nonseparable data, however, GD with large step sizes on the logistic regression problem can converge to stable periodic orbits~\citep{meng2024gradient}.
A separate line of work considers adaptive or increasing step sizes for separable logistic regression~\citep{axiotis2023gradient,zhang2025gradient,kale2026stretched}.

\paragraph{Multiclass logistic regression with large step sizes.}
For separable multiclass data, \citet{soudry2018implicit} and \citet{ravi2024implicit} showed that, for a sufficiently small fixed step size, the GD iterates converge in direction to the multiclass maximum-margin solution.
\citet{emmanouilidis2026sgd} study multiclass cross-entropy with arbitrary constant step sizes for minibatch SGD.
They prove an expected best-iterate convergence guarantee of order $\widetilde{O}((1+\eta^2)/(\gamma^2\eta t))$, extending the corresponding binary result for SGD from~\citep[Thm.~4]{wu2024large}.
Additionally, they develop a stochastic notion of stability under which the iterates enter a stable loss sublevel with high probability and subsequently return to it after stochastic excursions.
Their bound on the time to first enter the stable region grows linearly with the step size.
Our results instead concern deterministic full-batch GD: in the first part, we obtain an explicit $\mathcal O(\eta/\gamma^2)$ bound on the transition time, and subsequently show that, for sufficiently large step sizes, there is no dependence on $\eta$.

\paragraph{Perceptrons.}
The perceptron was introduced by~\citet{rosenblatt1958perceptron}, and its finite convergence on linearly separable data was proved by~\citet{novikoff1962convergence}.
The standard algorithm is online, while batch variants aggregate the corrections from all currently misclassified examples~\citep[\S1.6]{haykin2009neural}.
For multiclass classification, \citet{crammer2003ultraconservative} analyzed online perceptron-type algorithms and proved margin-dependent mistake bounds under linear separability.

A connection between large-step GD and batch perceptron dynamics was established by~\citet{tyurin2025logistic} for binary logistic regression: as the step size tends to infinity, the rescaled GD iterates converge to those of a batch perceptron.
Our $\eta \to \infty$ dynamics~\eqref{eq:multiclass-infinity-system} extends this connection to multiclass cross-entropy. 
In our case, the limiting dynamics is a full-batch multiclass perceptron that sums the corrections over all training examples, with the correction for each example divided equally among all classes with the largest score. 
Additionally, we make this connection quantitative for finite step sizes: for sufficiently large $\eta$, the GD iterates shadow the multiclass perceptron dynamics up to its termination time.
Under suitable conditions on the step size, at this same iteration GD has either entered the stable region or achieved loss below any prescribed $\varepsilon > 0$.

In their work showing convergence for GD on Fenchel-Young losses with any step size, \cite{bao2026any} also use a perceptron-type inequality, but for the original GD iterates $w_t$ rather than for rescaled iterates.
Their lower bound on $\langle w_t, w^*\rangle$ is obtained by recursively expanding the GD update (similar to our Lemma~\ref{lemma:good-direction-progress}), and they use the \emph{split optimization technique} from~\citep{wu2024large} for upper bounding $\| w_t \|$.
In contrast, our proofs do not use this technique: see Lemma~\ref{lemma:iterate-squared-norm}.

\section{A general perceptron-like dynamical system} \label{section:general-system}

The proofs of our main results build on perceptron-type arguments.
To isolate the parts of these arguments that we will reuse later, we first introduce a more general dynamical system.

Let $X \in \reals^{n\times d}$ be the matrix with rows $x_i^\top$, and $Y \in \reals^{n \times K}$ the one-hot label matrix as in~\eqref{eq:one-hot-labels}. 
Consider a map $\Pi \colon \reals^{K \times d} \to \reals^{n \times K}$ (which may depend on $X$) such that for all $i \in [n]$ and $k \in [K]$,
\begin{align*}
    \Pi(Z)_{ik} \geq 0 
    \quad\quad \textrm{and} \quad\quad
    \sum_{k=1}^K \Pi(Z)_{ik} = 1.
\end{align*}
That is, each row $\Pi(Z)_i$ is a probability vector over the $K$ classes: $\Pi(Z)_{iy_i}$ is the mass assigned to the correct class, while $\sum_{k \neq y_i} \Pi(Z)_{ik} = 1 - \Pi(Z)_{iy_i}$ is the total mass assigned to the other (incorrect) classes.
Next, we define the map $H \colon \reals^{K \times d} \to \reals^{n \times K}$ as 
\begin{equation*}
    H(Z) \coloneqq \frac{1}{n} \big( Y - \Pi(Z) \big).
\end{equation*}
The $i$-th row of $H(Z)$ compares the true label of the $i$-th data point with the probability vector $\Pi(Z)_i$.
Throughout this section, we study the following \emph{general batched perceptron dynamics}:
\begin{align} \label{eq:general-system} \tag{GBPD}
    \ZZ{t+1} =
    \ZZ{t} + H(\ZZ{t})^\top X,
    \qquad
    \ZZ{0} = 0 \in \reals^{K \times d}.
\end{align}
Each update is a linear combination of the data points, weighted by the difference between the true labels $Y$ and the predicted probability vectors $\Pi(\ZZ{t})$, hence the connection to the perceptron.
Unlike the original online perceptron, which updates using one data point at a time,~\eqref{eq:general-system} aggregates the contributions of all data points in each iteration.

For any $Z \in \reals^{K \times d}$, define the \emph{average mistake mass} as
\begin{align} \label{eq:avg-mistake-def}
    S(Z) \coloneqq 
    \frac{1}{nK} \sum_{i=1}^n \sum_{k \in [K], k \neq y_i} \Pi(Z)_{ik} = 
    \frac{1}{nK} \sum_{i=1}^n (1 - \Pi(Z)_{i y_i}).
\end{align}
Along the trajectory of~\eqref{eq:general-system}, we write $\SSS{t} \coloneqq S(\ZZ{t})$ when convenient.

The next lemma gives the \emph{perceptron inequality} for this system.
We will later apply it for two particular choices of $\Pi$.
In both cases, after further bounding the second term on the left-hand side, we obtain an upper bound on the cumulative average mistake mass $\sum_{t=0}^{T-1} S(\ZZ{t})$.
\begin{lemma} \label{lemma:main-perceptron}
    Let $(\ZZ{t})_{t \geq 0}$ be the trajectory of~\eqref{eq:general-system}.
    Under~\ref{ass:multiclass1} and~\ref{ass:multiclass2}, we have
    \begin{align*}
        2 K \sum_{t=0}^{T-1} S(\ZZ{t}) + 2 \sum_{t=0}^{T-1} \langle \ZZ{t}, H(\ZZ{t})^\top X \rangle_\mathrm{F} \geq \gamma^2 K^2 \left( \sum_{t=0}^{T-1} S(\ZZ{t}) \right)^2
        \quad \textrm{ for all } T \geq 1.
    \end{align*}
\end{lemma}
The rest of this section is devoted to the proof of Lemma~\ref{lemma:main-perceptron}.
The argument follows the standard perceptron strategy~\citep{novikoff1962convergence}: by recursively unrolling the~\eqref{eq:general-system} update rule, we first lower bound the inner product of the iterate with the separating matrix $W^*$ (Lemma~\ref{lemma:good-direction-progress}), and then upper bound the growth of $\| \ZZ{T} \|_\mathrm{F}$ (Lemma~\ref{lemma:iterate-squared-norm}).
The perceptron inequality then follows by combining these two estimates with the Cauchy--Schwarz inequality.

The next lemma shows that each update makes some progress in the direction of $W^*$.
\begin{lemma} \label{lemma:update-term-good-direction-progress}
    Under~\ref{ass:multiclass2}, the following holds:
    \begin{align*}
        \langle W^*, H(Z)^\top X \rangle_\mathrm{F} \geq
        \gamma K S(Z).
    \end{align*}
\end{lemma}
\begin{proof}
    Using $H(Z) = \frac{1}{n}(Y - \Pi(Z))$, we rewrite
    \begin{align} \label{eq:W_Z_inner_extra_term}
        \langle W^*, H(Z)^\top X\rangle_\mathrm{F} = 
        \sum_{k=1}^K \sum_{i=1}^n \langle w_k^*, x_i \rangle (H(Z)^\top)_{ki} = 
        \frac{1}{n} \sum_{k=1}^K \sum_{i=1}^n \langle w_k^*, x_i \rangle \; (Y_{ik} - \Pi_{ik}),
    \end{align}
    where we used the notation $\Pi = \Pi(Z)$, since there is no risk of confusion.
    Recall from Eq.~\eqref{eq:one-hot-labels} that $Y_{ik} = 1$ if $k = y_i$ and $Y_{ik} = 0$ otherwise.
    Using $\sum_{k=1}^K \Pi_{ik} = 1$, we can continue:
    \begin{align*}
        \langle W^*, H(Z)^\top X\rangle_\mathrm{F}
        &= \frac{1}{n}  \sum_{i=1}^n \Bigg( \sum_{k=1}^K \langle w_k^*, x_i \rangle \; Y_{ik} -
        \sum_{k=1}^K \langle w_k^*, x_i \rangle \; \Pi_{ik} \Bigg) \\
        &= \frac{1}{n} \sum_{i=1}^n \Bigg( \langle w_{y_i}^*, x_i \rangle - \sum_{k=1}^K \langle w_k^*, x_i \rangle \; \Pi_{ik} \Bigg) \\ 
        &=  \frac{1}{n} \sum_{i=1}^n \Bigg( \sum_{k=1}^K \langle w_{y_i}^*, x_i \rangle \; \Pi_{ik} - \sum_{k=1}^K \langle w_k^*, x_i \rangle \; \Pi_{ik} \Bigg) \\ 
        &= \frac{1}{n} \sum_{i=1}^n \sum_{k \in [K], k \neq y_i} \Pi_{ik} \langle w_{y_i}^* - w_k^*, x_i \rangle \\
        &\geq \frac{\gamma}{n} \sum_{i=1}^n \sum_{k \in [K], k \neq y_i} \Pi_{ik} 
        = \gamma K S(Z),
    \end{align*}
    where in the last step we used~\ref{ass:multiclass2} and the definition of the average mistake mass from Eq.~\eqref{eq:avg-mistake-def}.
\end{proof}

Summing the progress along the entire trajectory yields the next result.
\begin{lemma} \label{lemma:good-direction-progress}
    Under~\ref{ass:multiclass2}, the following inequality holds for all $T \geq 1$:
    \begin{align*}
        \langle W^*, \ZZ{T} \rangle_\mathrm{F} \geq \gamma K \sum_{t=0}^{T-1} S(\ZZ{t}).
    \end{align*}
\end{lemma}
\begin{proof}
    Using the matrix update of~\eqref{eq:general-system}, we obtain
    \begin{align} \label{eq:W_Z_inner_bound}
        \langle W^*, \ZZ{t+1}\rangle_\mathrm{F} = \langle W^*, \ZZ{t} \rangle_\mathrm{F} + \langle W^*, H(\ZZ{t})^\top X \rangle_\mathrm{F}.
    \end{align}
    Plugging in the bound from Lemma~\ref{lemma:update-term-good-direction-progress}, this becomes
    \begin{align*}
        \langle W^*, \ZZ{t+1} \rangle_\mathrm{F} \geq 
        \langle W^*, \ZZ{t} \rangle_\mathrm{F} + \gamma K \; S(\ZZ{t}).
    \end{align*}
    Summing this inequality over $t=0,\dots,T-1$, and using that $\ZZ{0} = 0$, gives the result.
\end{proof}

We next bound the growth of the squared norm of the iterate.
\begin{lemma} \label{lemma:iterate-squared-norm}
    Under~\ref{ass:multiclass1}, the squared norm of each iterate of~\eqref{eq:general-system} satisfies
    \begin{align*}
        \| \ZZ{T} \|_\mathrm{F}^2 \leq 
        2 K \sum_{t=0}^{T-1} S(\ZZ{t}) + 2 \sum_{t=0}^{T-1} \langle \ZZ{t}, H(\ZZ{t})^\top X \rangle_\mathrm{F}
        \quad\quad \textrm{ for all } T \geq 1 .
    \end{align*}
\end{lemma}
\begin{proof}
    Taking squared norms in~\eqref{eq:general-system}, we obtain
    \begin{align} \label{eq:Z_square_bound}
        \| \ZZ{t+1} \|_\mathrm{F}^2 = 
        \| \ZZ{t} \|_\mathrm{F}^2 + \| H(\ZZ{t})^\top X \|_\mathrm{F}^2 + 2 \langle \ZZ{t}, H(\ZZ{t})^\top X  \rangle_\mathrm{F}. 
    \end{align}
    Let us focus on the term $\| H(\ZZ{t})^\top X \|_\mathrm{F}^2$ first.
    Since $H(\ZZ{t}) = \frac{1}{n} (Y - \Pi(\ZZ{t}))$, we have 
    \begin{align} \label{eq:H_X_sq_norm_bound}
        \| H(\ZZ{t})^\top X \|_\mathrm{F}^2
        &= \sum_{k=1}^K \| (H(\ZZ{t})^\top X)_{k, :} \|^2 \notag \\ 
        &= \sum_{k=1}^K \left\| \frac{1}{n} \sum_{i=1}^n (Y_{ik} - \Pi(\ZZ{t})_{ik}) \; x_i \right\|^2 \notag \\ 
        &\leq \frac{1}{n} \sum_{k=1}^K \sum_{i=1}^n (Y_{ik} - \Pi(\ZZ{t})_{ik})^2 \; \|x_i\|^2 \notag \\ 
        &\leq \frac{1}{n} \sum_{i=1}^n \sum_{k=1}^K (Y_{ik} - \Pi(\ZZ{t})_{ik})^2,
    \end{align}
    where in the third step we applied Jensen's inequality to the function $v \mapsto \| v \|^2$, and in the last one we used~\ref{ass:multiclass1}.
    We now relate this quantity to the average mistake mass $S(\ZZ{t})$ from Eq.~\eqref{eq:avg-mistake-def}.
    To this end, first use that the rows of $Y$ are one-hot encodings for the labels $y_i$, then the fact that the rows of $\Pi(\ZZ{t})$ are probability vectors, so that $1 - \Pi(\ZZ{t})_{i y_i} = \sum_{k \in [K], k\neq y_i} \Pi(\ZZ{t})_{ik}$.
    Therefore,
    \begin{align*}
        \sum_{k=1}^K (Y_{ik} - \Pi(\ZZ{t})_{ik})^2
        &= (1 - \Pi(\ZZ{t})_{i y_i})^2 + \sum_{k \in [K], k \neq y_i} (\Pi(\ZZ{t})_{ik})^2 \\ 
        &= \left( \sum_{k \in [K], k \neq y_i} \Pi(\ZZ{t})_{ik} \right)^2 + \sum_{k \in [K], k \neq y_i} (\Pi(\ZZ{t})_{ik})^2 \\ 
        &\leq 2 \left( \sum_{k \in [K], k \neq y_i} \Pi(\ZZ{t})_{ik} \right)^2 \\ 
        &\leq 2 \sum_{k \in [K], k \neq y_i} \Pi(\ZZ{t})_{ik},
    \end{align*}
    where in the last step we used that $\sum_{k \in [K], k \neq y_i} \Pi(\ZZ{t})_{ik} \leq 1$.
    Plugging this in Eq.~\eqref{eq:H_X_sq_norm_bound} yields
    \begin{align*}
        \| H(\ZZ{t})^\top X \|_\mathrm{F}^2
        \leq \frac{2}{n} \sum_{i=1}^n \sum_{k \in [K], k \neq y_i} \Pi(\ZZ{t})_{ik}
        = 2 K S(\ZZ{t}),
    \end{align*}
    and with this, Eq.~\eqref{eq:Z_square_bound} becomes
    \begin{align*}
        \| \ZZ{t+1} \|_\mathrm{F}^2 \leq
        \| \ZZ{t} \|_\mathrm{F}^2 + 2 K S(\ZZ{t}) + 2 \langle \ZZ{t}, H(\ZZ{t})^\top X \rangle_\mathrm{F}.
    \end{align*}
    To conclude, unroll this from $t=T-1$ to $t=0$ and use $\ZZ{0} = 0$.
\end{proof}

The proof of Lemma~\ref{lemma:main-perceptron} is now immediate.
\begin{proof}[Proof of Lemma~\ref{lemma:main-perceptron}]
    By Lemma~\ref{lemma:good-direction-progress}, for all $T \geq 1$, we have
    \begin{align*}
        \langle W^*, \ZZ{T} \rangle_\mathrm{F} \geq \gamma K \sum_{t=0}^{T-1} S(\ZZ{t}).
    \end{align*}
    Since both sides of the inequality are nonnegative, we can square them.
    Via Cauchy--Schwarz and using $\| W^* \|_\mathrm{F} \leq 1$ (from~\ref{ass:multiclass2}), we conclude with Lemma~\ref{lemma:iterate-squared-norm} that
    \begin{align*}
        \gamma^2 K^2 \left( \sum_{t=0}^{T-1} S(\ZZ{t}) \right)^2 \leq
        \langle W^*, \ZZ{T} \rangle_\mathrm{F}^2 \leq 
        \| \ZZ{T} \|_\mathrm{F}^2 \leq 
        2 K \sum_{t=0}^{T-1} S(\ZZ{t}) + 2 \sum_{t=0}^{T-1} \langle \ZZ{t}, H(\ZZ{t})^\top X \rangle_\mathrm{F},
    \end{align*}
    which is the claimed inequality.
\end{proof}

\section{Two-phase convergence of large-step GD} \label{section:two-phase}

In this section, we prove a quantitative version of Theorem~\ref{thm:multiclass-two-phase-informal}, with explicit bounds on both the transition time to the stable region, and the value of the loss therein.

\begin{theorem} \label{thm:multiclass-two-phase}
    Suppose~\ref{ass:multiclass1} and~\ref{ass:multiclass2} hold.
    If we optimize $f$ as in~\eqref{eq:multiclass-loss} using GD initialized at $\WWeta{0} = 0$ with constant step size $\eta \geq \frac{n}{2}$, then:
    \begin{enumerate}[label=(\roman*)]
        \item (\textbf{Transition time.})
        The transition time to the stable region $\tau$~\eqref{eq:tau-def} satisfies
        \begin{align*}
            \tau \leq
            \frac{8}{\gamma^2} \left( \eta + \log\bigl( 4 \eta e (K-1) \bigr) \right).
        \end{align*}

        \item (\textbf{Loss decrease in the stable region.})
        For all $t \geq \tau$, the iterates remain in the stable region and the loss satisfies
        \begin{align*}
            f(\WWeta{t})
            \leq
            \frac{1}{ \frac{1}{f(\WWeta{\tau})} + \frac{\eta \gamma^2}{8}(t - \tau) } .
        \end{align*}
    \end{enumerate}
\end{theorem}

Choosing the step size in terms of the iteration budget yields an accelerated $\mathcal O(1 / T^2)$ rate.
\begin{corollary} \label{cor:multiclass-accelerated-rate}
    Consider GD with step size $\eta > 0$ for the multiclass cross-entropy loss~\eqref{eq:multiclass-loss} under Assumptions~\ref{ass:multiclass1} and~\ref{ass:multiclass2} and take $\WWeta{0} = 0$.
    Assume the iteration budget satisfies
    \begin{align*}
        T \geq \frac{1}{\gamma^2} \max\left\{ 16n, 64 \log(4e(K-1)) \right\},
    \end{align*}
    and take $\eta = \frac{\gamma^2 T}{32}$.
    Then, the transition time (Eq.~\eqref{eq:tau-def}) satisfies $\tau \leq \frac{T}{2}$ and $f(\WWeta{T}) \leq \frac{512}{\gamma^4 T^2}$.
\end{corollary}

To prove these results, we first derive the GD dynamics for the multiclass cross-entropy loss $f$.
In Section~\ref{subsection:soft-perceptron}, we show that, after a change of variables, the resulting dynamical system is a general batched perceptron~\eqref{eq:general-system}, and so we can apply the general results established in the previous section.
Combining this with a lower bound on the average mistake mass $\SSS{t}$ in Subsection~\ref{subsection:transition-time} yields the result about the transition time.
Then, in Section~\ref{subsection:loss-decay}, we use a separate descent argument in the stable region to prove the loss rate from the second part of Theorem~\ref{thm:multiclass-two-phase}, as well as the corollary.

\subsection{Rescaled dynamics: the soft perceptron} \label{subsection:soft-perceptron}

Fix a step size $\eta > 0$ and consider the GD iterates $\WWeta{t}$ optimizing the loss $f$ as in~\eqref{eq:multiclass-loss}, initialized at $\WWeta{0} = 0$.
A direct computation (Appendix~\ref{appendix:gd-dynamics-lemma-proof}) shows that the GD update takes the matrix form
\begin{align} \label{eq:gd-dynamics}
    \WWeta{t+1} = \WWeta{t} + \frac{\eta}{n} \big(Y - P(\WWeta{t}) \big)^\top X.
\end{align}
To apply the general perceptron argument from Section~\ref{section:general-system}, we define the \emph{rescaled iterates}
\begin{align*}
    \ZZeta{t} \coloneqq \frac{1}{\eta} \WWeta{t}.
\end{align*}
Dividing the GD update~\eqref{eq:gd-dynamics} by $\eta$ yields the following update rule for $\ZZeta{t}$:
\begin{align} \label{eq:multiclass-eta-system} \tag{$\mathrm{SP}_\eta$}
    \ZZeta{t+1} =
    \ZZeta{t} + \frac{1}{n} \big(Y - P(\eta \ZZeta{t})\big)^\top X,
    \qquad
    \ZZeta{0} = 0.
\end{align}
We call~\eqref{eq:multiclass-eta-system} the \emph{soft perceptron dynamics}.
Observe that this corresponds to~\eqref{eq:general-system} with
\begin{align*}
    \Pi(Z) = P(\eta Z),
    \qquad \textrm{and} \qquad
    H(Z) = \frac{1}{n}\bigl(Y-P(\eta Z)\bigr).
\end{align*}
For this choice of $\Pi$, the \emph{average mistake mass} at iteration $t$ is the total softmax probability assigned to the incorrect classes, which we denote by
\begin{align} \label{eq:avg-mistake-gd}
    \SSS{t} \coloneqq
    S(\ZZeta{t}) = 
    \frac{1}{nK} \sum_{i=1}^n \sum_{k \in [K], k \neq y_i} P(\eta \ZZeta{t} )_{ik}
\end{align}

\subsection{Specializing the perceptron bound to the~\eqref{eq:multiclass-eta-system} dynamics}

In this section, we specialize the perceptron-type bound in Lemma~\ref{lemma:main-perceptron} to the system~\eqref{eq:multiclass-eta-system}.
Recall that this bound controls the cumulative average mistake mass $\SSS{t}$ (defined in Eq.~\eqref{eq:avg-mistake-def}).
By further controlling the cross term $\langle \ZZeta{t}, H(\ZZeta{t})^\top X \rangle_\mathrm{F}$ through $\SSS{t}$, we obtain the following result.
\begin{lemma} \label{lemma:finite-eta-perceptron-inequality}
    Let $(\ZZeta{t})_{t \geq 0}$ be the trajectory of the soft perceptron dynamics~\eqref{eq:multiclass-eta-system}.
    Under Assumptions~\ref{ass:multiclass1} and~\ref{ass:multiclass2}, for every $T \geq 1$, we have
    \begin{align*}
        2 \sum_{t=0}^{T-1} \SSS{t}
        + \frac{2}{\eta} \sum_{t=0}^{T-1} \SSS{t} \log\left(\frac{e(K-1)}{K \SSS{t}} \right)
        \geq
        \gamma^2 K \left( \sum_{t=0}^{T-1} \SSS{t} \right)^2.
    \end{align*}
\end{lemma}

To obtain this inequality, we combine Lemma~\ref{lemma:main-perceptron} with the next two results.
The first one bounds the cross term from Lemma~\ref{lemma:main-perceptron} by the average entropy of the softmax probabilities, and the second one further bounds this average entropy in terms of the average mistake mass $\SSS{t}$.
Substituting these bounds into the general perceptron-type inequality yields Lemma~\ref{lemma:finite-eta-perceptron-inequality}.

For a probability vector $p = (p_1,\ldots,p_K)$, we denote its entropy by $\Ent(p)$, defined as
\begin{align*}
    \Ent(p) \coloneqq -\sum_{k=1}^K p_k \log p_k.
\end{align*}
\begin{lemma} \label{lemma:soft-perceptron-cross-term}
    Let $(\ZZeta{t})_{t \geq 0}$ be the iterates generated by the soft perceptron dynamics~\eqref{eq:multiclass-eta-system}.
    Then,
    \begin{align*}
        \langle \ZZeta{t}, H(\ZZeta{t})^\top X \rangle_\mathrm{F}
        \leq \frac{1}{n\eta} \sum_{i=1}^n \Ent(P(\eta \ZZeta{t})_i)
        \quad\quad \textrm{ for all } t \geq 0,
    \end{align*}
    where $P(\eta \ZZeta{t})_i$ denotes the $i$-th row of $P(\eta \ZZeta{t})$.
\end{lemma}

\begin{proof}
    Letting $P = P(\eta \ZZeta{t})$, we have $H(\ZZeta{t}) = \frac{1}{n}(Y-P)$, and so
    \begin{align} \label{eq:cross-term}
        \langle \ZZeta{t}, H(\ZZeta{t})^\top X \rangle_\mathrm{F} =
        \frac{1}{n} \langle \ZZeta{t} X^\top, (Y - P)^\top \rangle_\mathrm{F} = 
        \frac{1}{n} \sum_{i=1}^n \sum_{k=1}^K \langle \Zzeta{t}_k, x_i \rangle (Y - P)_{ik}.
    \end{align}
    Using that for all $i \in [n]$ the one-hot matrix satisfies $Y_{i y_i} = 1$ and $Y_{ik} = 0$ for all $k \neq y_i$, and also using that each row of $P$ is a probability vector, we develop this further
    \begin{align} \label{eq:Z_HX_inner_soft}
        \sum_{i=1}^n \sum_{k=1}^K \langle \Zzeta{t}_k, x_i \rangle (Y - P)_{ik}
        &= \sum_{i=1}^n \sum_{k=1}^K \langle \Zzeta{t}_k, x_i \rangle Y_{ik} - 
            \sum_{i=1}^n \sum_{k=1}^K \langle \Zzeta{t}_k, x_i \rangle P_{ik} \notag \\ 
        &= \sum_{i=1}^n \langle \Zzeta{t}_{y_i}, x_i \rangle - 
        \sum_{i=1}^n \sum_{k=1}^K \langle \Zzeta{t}_k, x_i \rangle P_{ik} \notag \\
        &= \sum_{i=1}^n \sum_{k=1}^K \langle \Zzeta{t}_{y_i}, x_i \rangle P_{ik} - 
        \sum_{i=1}^n \sum_{k=1}^K \langle \Zzeta{t}_k, x_i \rangle P_{ik} \notag \\
        &= \sum_{i=1}^n \sum_{k=1}^K \langle \Zzeta{t}_{y_i} - \Zzeta{t}_k, x_i \rangle P_{ik}.
    \end{align}
    Recall from the definition of $P_{ik} = P(\eta \ZZeta{t})_{ik}$ (Eq.~\eqref{eq:multiclass-loss}) that
    \begin{align*}
        P_{ik}
        = \frac{\exp(\eta \langle \Zzeta{t}_k, x_i \rangle)}{\sum_{l=1}^K \exp(\eta \langle \Zzeta{t}_l, x_i \rangle)}
        = \frac{\exp(\eta \langle \Zzeta{t}_k - \Zzeta{t}_{y_i}, x_i \rangle)}{\sum_{l=1}^K \exp(\eta \langle \Zzeta{t}_l - \Zzeta{t}_{y_i}, x_i \rangle)}.
    \end{align*}
    Taking logs and rearranging, we obtain 
    \begin{align*}
        \eta \langle \Zzeta{t}_{y_i} - \Zzeta{t}_k, x_i \rangle 
        &= - \log(P_{ik}) - \log\left( \sum_{l=1}^K \exp(\eta \langle \Zzeta{t}_l - \Zzeta{t}_{y_i}, x_i \rangle) \right) \\
        &= - \log(P_{ik}) - \log\left( 1 + \sum_{l \neq y_i} \exp(\eta \langle \Zzeta{t}_l - \Zzeta{t}_{y_i}, x_i \rangle) \right) \\
        &\leq - \log(P_{ik}).
    \end{align*}
    Combining this with Eqs.~\eqref{eq:cross-term} and~\eqref{eq:Z_HX_inner_soft}, we obtain
    \begin{align*}
        \langle \ZZeta{t}, H(\ZZeta{t})^\top X \rangle_\mathrm{F}
        &= \frac{1}{n \eta} \sum_{i=1}^n \sum_{k=1}^K
        P_{ik} \, \eta \langle \Zzeta{t}_{y_i} - \Zzeta{t}_k, x_i \rangle \\
        &\leq - \frac{1}{n \eta} \sum_{i=1}^n \sum_{k=1}^K P_{ik} \log(P_{ik}) 
        = \frac{1}{n \eta} \sum_{i=1}^n \Ent( P(\eta \ZZeta{t})_i),
    \end{align*}
    as claimed.
\end{proof}

The previous lemma bounds the cross term by the average entropy of the softmax probabilities.
We next control this entropy in terms of the average mistake mass $\SSS{t}$.

\begin{lemma} \label{lemma:soft-perceptron-type-bound-refined}
    For every $t \geq 0$, the average entropy of the softmax probabilities satisfies
    \begin{align*}
        \frac{1}{n} \sum_{i=1}^n \Ent(P(\eta \ZZeta{t})_i) \leq
        K \SSS{t} \, \log\left( \frac{e (K-1)}{K \SSS{t}} \right).
    \end{align*}
\end{lemma}
\begin{proof}
    Let $P = P(\eta \ZZeta{t})$.
    In what follows, we drop the superscript $t$ from all notation.
    Define
    \begin{align*}
        m_i \coloneqq
        1 - P_{i y_i} =
        \sum_{k \neq y_i} P_{ik} \in (0, 1)
        \quad \textrm{ and } \quad
        \tilde P_{ik} \coloneqq \frac{P_{ik}}{m_{i}} \textrm{ for all } k \neq y_i.
    \end{align*}
    Observe that $\tilde P_i$ is a probability vector, since its elements are non-negative and $\sum_{k \neq y_i} \tilde P_{ik} = 1$.
    Then, we decompose the entropy as
    \begin{align*}
        \Ent(P_i)
        &= - P_{i y_i} \log(P_{i y_i}) - \sum_{k \neq y_i} P_{ik} \log(P_{ik}) \\
        &= - P_{i y_i} \log(P_{i y_i}) - \sum_{k \neq y_i} m_i \tilde P_{ik} \log(m_i \tilde P_{ik}) \\
        &= -(1-m_i) \log(1-m_i) - m_i \left( \sum_{k \neq y_i} \tilde P_{ik} \log(m_i) + \sum_{k \neq y_i} \tilde P_{ik} \log(\tilde P_{ik})\right) \\
        &= -(1-m_i) \log(1-m_i) - m_i \log(m_i) + m_i \Ent(\tilde P_i) \\
        &= h(m_i) + m_i \Ent(\tilde P_i)
    \end{align*}
    where we defined the binary entropy $h \colon [0, 1] \to \reals$ as $h(p) = -p\log(p) - (1-p)\log(1-p)$.
    This gives
    \begin{align} \label{eq:entropy-sum}
        \frac{1}{n} \sum_{i=1}^n \Ent(P_i) = 
        \frac{1}{n} \sum_{i=1}^n h(m_i) + \frac{1}{n} \sum_{i=1}^n m_i \Ent(\tilde P_i).
    \end{align}
    Recall from Eq.~\eqref{eq:avg-mistake-gd} that $S(Z) = \frac{1}{nK} \sum_{i=1}^n \sum_{k \neq y_i} P_{ik} = \frac{1}{nK} \sum_{i=1}^n m_i$.
    Since $h$ is concave, applying Jensen's inequality we obtain
    \begin{align*}
        \frac{1}{n} \sum_{i=1}^n h(m_i) \leq 
        h \left( \frac{1}{n} \sum_{i=1}^n m_i \right) = 
        h \left( K S(Z) \right).
    \end{align*}
    Thus, plugging this back in Eq.~\eqref{eq:entropy-sum} we have shown that
    \begin{align*}
        \frac{1}{n} \sum_{i=1}^n \Ent(P_i) \leq
        h(K S(Z)) + \frac{1}{n} \sum_{i=1}^n m_i \Ent(\tilde P_i).
    \end{align*}
    As $\tilde P_i$ is a discrete random variable with $K-1$ possible outcomes, its entropy is bounded above as $\Ent(\tilde P_i) \leq \log(K-1)$, and using again that $\frac{1}{n} \sum_{i=1}^n m_i = K S(Z)$, we get
    \begin{align} \label{eq:entropy-sum-refined}
        \frac{1}{n} \sum_{i=1}^n \Ent(P_i)
        &\leq h(K S(Z)) + \log(K-1) \frac{1}{n} \sum_{i=1}^n m_i \notag \\
        &= h(K S(Z)) + \log(K-1) K S(Z).
    \end{align}
    To obtain the final result, it remains to bound $h(K S(Z))$.
    Using that $-(1-p) \log(1-p) \leq p$ for all $p \in [0, 1]$, we obtain
    \begin{align*}
        h(p)
        = -p\log(p) - (1-p)\log(1-p) 
        \leq -p\log(p) + p 
        = p \log\left( \frac{e}{p} \right).
    \end{align*}
    Taking $p = K S(Z) \in [0, 1]$ and plugging the result back in Eq.~\eqref{eq:entropy-sum-refined}, we conclude
    \begin{align*}
        \frac{1}{n} \sum_{i=1}^n \Ent(P_i) \leq 
        K S(Z) \log\left( \frac{e(K-1)}{K S(Z)} \right).
    \end{align*}
    This proves Lemma~\ref{lemma:soft-perceptron-type-bound-refined}.
\end{proof}

We can now combine these results to prove Lemma~\ref{lemma:finite-eta-perceptron-inequality}.
\begin{proof}[Proof of Lemma~\ref{lemma:finite-eta-perceptron-inequality}]
    Applying Lemma~\ref{lemma:main-perceptron} with $\Pi(Z) = P(\eta Z)$ gives
    \begin{align*}
        2 K \sum_{t=0}^{T-1}\SSS{t} + 2 \sum_{t=0}^{T-1} \langle \ZZeta{t},H(\ZZeta{t})^\top X\rangle_\mathrm{F} \geq
        \gamma^2 K^2 \left(\sum_{t=0}^{T-1} \SSS{t} \right)^2.
    \end{align*}
    By Lemmas~\ref{lemma:soft-perceptron-cross-term} and~\ref{lemma:soft-perceptron-type-bound-refined},
    \begin{align*}
        \langle \ZZeta{t}, H(\ZZeta{t})^\top X \rangle_\mathrm{F} \leq
        \frac{K}{\eta} \SSS{t} \log\left( \frac{e(K-1)}{K\SSS{t}} \right).
    \end{align*}
    Substituting this bound into the previous inequality and dividing by $K$ gives the result.
\end{proof}

\subsection{Transition to the stable region} \label{subsection:transition-time}

We now prove the first part of Theorem~\ref{thm:multiclass-two-phase}.
The key ingredient is a lower bound on the average mistake mass $\SSS{t}$ while the iterates are in the unstable region.
Combined with Lemma~\ref{lemma:finite-eta-perceptron-inequality}, this yields an upper bound on the length of the unstable phase.

We begin with a general lower bound on the average mistake mass in terms of the loss.
The result is the multiclass analogue of Lemma~C.1 in~\citep{crawshaw2026tight}, although our proof uses a different argument.
\begin{lemma} \label{lemma:mistake-mass}
    For every $Z \in \reals^{K \times d}$, the average mistake mass (Eq.~\eqref{eq:avg-mistake-gd}) satisfies
    \begin{align*}
        S(Z) \geq
        \frac{1-\exp(-nf(\eta Z))}{nK}.
    \end{align*}
\end{lemma}
\begin{proof}
    Let $P = P(\eta Z)$, and define $l_i \coloneqq - \log(P_{iy_i})$ and $m_i \coloneqq \sum_{k \neq y_i} P_{ik}$.
    Recall that $f(\eta Z) = \frac{1}{n} \sum_{i=1}^n l_i$, and so
    \begin{align} \label{eq:exp_nF}
        \exp(-n f(\eta Z)) = 
        \prod_{i=1}^n \exp(-l_i) = 
        \prod_{i=1}^n P_{i y_i} = 
        \prod_{i=1}^n (1 - m_i).
    \end{align}
    Since $S(Z) = \frac{1}{nK} \sum_{i=1}^n \sum_{k \neq y_i} P_{ik} = \frac{1}{nK} \sum_{i=1}^n m_i$, the goal is to relate the above product to the sum $\sum_{i=1}^n m_i$.
    To do this, we use the Weierstrass product inequality, which gives that since $m_i \in [0, 1]$, $\prod_{i=1}^n (1-m_i) \geq 1 - \sum_{i=1}^n m_i$.
    Thus, we obtain
    \begin{align*}
        S(Z)
        = \frac{1}{nK} \sum_{i=1}^n m_i
        \geq \frac{1}{nK} \left( 1 - \prod_{i=1}^n (1 - m_i) \right)
        = \frac{1}{nK} \left( 1 - \exp(-n f(\eta Z)) \right),
    \end{align*}
    where the last step follows from Eq.~\eqref{eq:exp_nF}
\end{proof}

When the iterate is in the unstable region, $\SSS{t}$ can be lower bounded further, as follows.
\begin{lemma} \label{lemma:large-f-implies-mistake-mass}
    Assume that $f(\eta Z) \geq \frac{1}{2 \eta}$.
    Then, $S(Z) \geq \frac{1}{nK} (1 - \exp(-\frac{n}{2\eta})) $.
    Additionally, if $\eta \geq n / 2$, then $ S(Z) \geq \frac{1}{4K \eta}$.
\end{lemma}
\begin{proof}
    First, plugging $f(\eta Z) \geq \frac{1}{2\eta}$ into Lemma~\ref{lemma:mistake-mass} gives
    \begin{align*}
        S(Z) \geq
        \frac{1}{nK} \left( 1 - \exp\left( -\frac{n}{2\eta} \right) \right).
    \end{align*}
    Next, assume $\eta \geq n/2$, so that $n / (2\eta) \in [0, 1]$.
    Using $1 - e^{-x} \geq x/2$ for $x \in [0, 1]$, we conclude that $S(Z) \geq \frac{1}{4 K \eta}$.
\end{proof}

Combining the preceding lemmas yields the bound on the time required to enter the stable region.
\begin{proof}[Proof of part (i) of Theorem~\ref{thm:multiclass-two-phase}.]
    Let
    \begin{align*}
        \tau_{\mathrm{ub}}
        \coloneqq
        \frac{8 \left( \eta + \log\left( 4e(K-1)\eta \right) \right)}{\gamma^2}.
    \end{align*}
    Suppose, for contradiction, that $\tau > \tau_{\mathrm{ub}}$, and let $ T \coloneqq \lfloor \tau_{\mathrm{ub}} \rfloor + 1 $.
    Then $T \leq \tau$.
    We apply Lemma~\ref{lemma:finite-eta-perceptron-inequality} with this $T$, which gives
    \begin{align} \label{eq:soft-perceptron-bound-nice}
        2 \sum_{t=0}^{T-1} \SSS{t}
        + \frac{2}{\eta}
        \sum_{t=0}^{T-1}  \SSS{t} \log\left( \frac{e (K-1)}{K \SSS{t}}\right)
        \geq
        \gamma^2 K
        \left(
            \sum_{t=0}^{T-1} \SSS{t}
        \right)^2.
    \end{align}
    For all $t \in \{0, \dots, T-1\}$, since $T \leq \tau$, from the definition of $\tau$ we know that $f(\WWeta{t}) \geq \frac{1}{2\eta}$.
    For each such $t$, Lemma~\ref{lemma:large-f-implies-mistake-mass} gives $\SSS{t} \geq \frac{1}{4K\eta}$.
    This implies that $ \log\left( \frac{e (K-1)}{K \SSS{t}}\right) \leq \log\left( 4e(K-1)\eta \right)$, so plugging this in Eq.~\eqref{eq:soft-perceptron-bound-nice} we obtain
    \begin{align*}
        2 \sum_{t=0}^{T-1} \SSS{t} + \frac{2}{\eta} \log\left( 4e(K-1)\eta \right) \sum_{t=0}^{T-1}  \SSS{t} \geq
        \gamma^2 K \left( \sum_{t=0}^{T-1} \SSS{t} \right)^2 .
    \end{align*}
    After rearranging and using again $\SSS{t} \geq \frac{1}{4 K \eta}$ for all $0 \leq t \leq T-1$, we further obtain
    \begin{align*}
        \frac{1}{\gamma^2 K} \left(2 + \frac{2}{\eta} \log\left( 4e(K-1)\eta \right) \right)
        \geq \sum_{t=0}^{T-1} \SSS{t}
        \geq \frac{T}{4 K \eta}.
    \end{align*}
    This shows that $T$ satisfies 
    \begin{align*}
        T \leq 
        \frac{4 \eta}{\gamma^2} \left(2 + \frac{2}{\eta} \log\left( 4e(K-1)\eta \right) \right) = 
        \frac{8 \left( \eta + \log\left( 4e(K-1)\eta \right) \right)}{\gamma^2}
        = \tau_{\mathrm{ub}},
    \end{align*}
    which contradicts $T = \lfloor \tau_{\mathrm{ub}} \rfloor + 1 > \tau_{\mathrm{ub}}$.
    Therefore, $\tau \leq \tau_{\mathrm{ub}}$, as claimed in Theorem~\ref{thm:multiclass-two-phase}, part (i).   
\end{proof}

\subsection{Loss decay in the stable region} \label{subsection:loss-decay}

This section proves part (ii) of Theorem~\ref{thm:multiclass-two-phase}, following the argument from the proof of Theorem~2.1 in~\citet{crawshaw2026tight}.

By definition, $\WWeta{\tau}$ is in the stable region.
By Lemma~\ref{lemma:cross-entropy-hessian-self-bound}, it holds that $\nabla^2 f(W) \preceq 2 f(W) I$ for all $W$.
Then, Lemma~\ref{lemma:descent-lemma-if-stable} gives that the stable region is forward invariant and that, for all $t \geq \tau$,
\begin{align*}
    f(\WWeta{t+1}) - f(\WWeta{t}) \leq
    -\frac{\eta}{2} \| \nabla f(\WWeta{t}) \|_\mathrm{F}^2.
\end{align*}
Combining this with Lemma~\ref{lemma:gradF-lower-bounded-by-F-if-stable}, which gives $\| \nabla f(\WWeta{t}) \|_\mathrm{F} \geq \frac{\gamma}{2} f(\WWeta{t})$ in the stable region, gives
\begin{align*}
    f(\WWeta{t+1}) - f(\WWeta{t}) \leq
    -\frac{\eta}{8} \gamma^2 f(\WWeta{t})^2
    \qquad \textrm{ for all } t \geq \tau.
\end{align*}
Dividing this by $f(\WWeta{t}) f(\WWeta{t+1})$ and using $f(\WWeta{t+1}) \leq f(\WWeta{t})$, we obtain
\begin{align*}
    \frac{1}{f(\WWeta{t+1})} - \frac{1}{f(\WWeta{t})} \geq
    \frac{\eta}{8} \gamma^2 \frac{f(\WWeta{t})}{f(\WWeta{t+1})} \geq
    \frac{\eta}{8} \gamma^2.
\end{align*}
Unrolling this from $\tau$ to $t-1$ yields
\begin{align*}
    \frac{1}{f(\WWeta{t})} - \frac{1}{f(\WWeta{\tau})}
    \geq
    (t-\tau) \frac{\eta}{8} \gamma^2.
\end{align*}
The second part of Theorem~\ref{thm:multiclass-two-phase} follows by rearranging.

Finally, Corollary~\ref{cor:multiclass-accelerated-rate} follows by taking $\eta = \frac{\gamma^2 T}{32}$.
Under the stated assumptions on $T$, part (i) gives $\tau \leq T/2$, and part (ii) then yields $f(\WWeta{T}) \leq \frac{512}{\gamma^4 T^2}$ (see Appendix~\ref{appendix:corollary-proof} for details).

\section{Step-size-independent transition time} \label{section:independent-step-size}

The goal of this section is to prove a quantitative version of Theorem~\ref{thm:multiclass-main-up-informal}.
We do so by comparing the soft perceptron dynamics~\eqref{eq:multiclass-eta-system} with a dynamical system obtained by taking $\eta \to \infty$.
In Section~\ref{subsection:hard-perceptron-limit}, we show that this limit is indeed well defined and is given by the following \emph{hard perceptron} dynamics:
\begin{align} \label{eq:multiclass-infinity-system} \tag{HP}
    \ZZinf{t+1} =
    \ZZinf{t} + \frac{1}{n} \bigl(Y - Q(\ZZinf{t})\bigr)^\top X,
    \qquad
    \ZZinf{0}=0,
\end{align}
where $Q \colon \reals^{K \times d} \to \reals^{n \times K}$ is a \emph{row-hardmax} operation given by
\begin{align} \label{eq:multiclass-Q-formula}
    Q(Z)_{ik} = \big(\hardmax(Z x_i)\big)_k,
    \quad \textrm{ with } \quad 
    \big(\hardmax(s)\big)_k \coloneqq
    \frac{\mathbbm{1}_{\{k \in \argmax{\ell \in \{1,\dots,K\}} s_\ell \}}}{\left| \argmax{\ell \in \{1,\dots,K\}} s_\ell \, \right|}.
\end{align}

We call a matrix $Z \in \reals^{K \times d}$ a \emph{perfect classifier} if, for every training example, the correct class has a strictly larger inner product with the input than every incorrect class:
\begin{align} \label{eq:perfect-classifier-def}
    \langle z_{y_i}, x_i \rangle > \langle z_k, x_i \rangle
    \qquad\qquad \textrm{ for all } i \in [n] \textrm{ and all } k \neq y_i.
\end{align}
If an iterate $Z$ is a perfect classifier, then $Q(Z) = Y$ and so the hard perceptron dynamics terminate.
In Lemma~\ref{lemma:multiclass-perceptron-finite-time}, we show that the hard perceptron dynamics~\eqref{eq:multiclass-infinity-system} reach a perfect classifier after 
\begin{align} \label{eq:T-def}
    T \leq \left\lfloor \frac{4 n}{\gamma^2} \right\rfloor + 1
\end{align}
iterations.
This is a batch analogue of the classical convergence guarantees for the online multiclass perceptron \citep{crammer2003ultraconservative}.
In Section~\ref{section:shadowing-prop}, we then show that for sufficiently large $\eta$, the soft perceptron trajectory ``shadows'' the limiting trajectory up to this time.
More specifically, the \emph{finite-$\eta$ scores} $ \Seta{t}_{ik} \coloneqq \langle \Zzeta{t}_k, x_i \rangle$ remain close enough to the \emph{limiting scores} $ \Sinf{t}_{ik} \coloneqq \langle \Zzinf{t}_k, x_i \rangle$.
Define the set of ``winner'' classes (attaining the largest \emph{limiting} score) by
\begin{align*}
    \winner{t}_i \coloneqq
    \argmax{k\in[K]} \; \Sinf{t}_{ik}.
\end{align*}
We then define the smallest positive gap between a maximizing and a nonmaximizing class along the first $T$ limiting iterates as
\begin{align} \label{eq:delta-def}
    \delta \coloneqq
    \min
    \left\{
        \Sinf{t}_{i\ell}-\Sinf{t}_{ik} \;
        \middle| \;
        0 \leq t \leq T, \;
        i\in[n],\;
        \ell\in\winner{t}_i,\;
        k\notin\winner{t}_i
    \right\}.
\end{align}
The minimum is taken over a finite, nonempty set of positive values, and therefore $\delta>0$.
Moreover, observe that $\delta$ is \emph{determined entirely by the~\eqref{eq:multiclass-infinity-system} trajectory}, and therefore depends only on the data $X$ and labels $Y$.
In particular, it is independent of the GD step size $\eta$.

Although the shadowing result controls the soft perceptron dynamics only up to time $T$, choosing a sufficiently large $\eta$ ensures that, by this time, the iterates have already entered the stable region, despite the fact that the stable sublevel set shrinks as $\eta$ grows.
Alternatively, for any prescribed $\varepsilon > 0$, the step size can be chosen so that the loss at time $T$ is at most $\varepsilon$.
The following theorem makes these statements precise.
\begin{theorem} \label{thm:multiclass-main}
    Fix a dataset $\{ (x_i, y_i) \}_{i = 1}^n$ satisfying~\ref{ass:multiclass1} and~\ref{ass:multiclass2}.
    Let $T$ be the termination time of the hard perceptron~\eqref{eq:multiclass-infinity-system} on this dataset (well defined by Lemma~\ref{lemma:multiclass-perceptron-finite-time}), and let $\delta > 0$ be the gap defined in~\eqref{eq:delta-def}.
    Define $r \coloneqq \max\{1, T/\delta\}$.
    If we optimize the loss $f$ as in~\eqref{eq:multiclass-loss} using GD initialized at $\WWeta{0} = 0$ with constant step size $\eta > 0$, then:
    \begin{enumerate}[label=(\roman*)]
        \item (\textbf{Transition time.})
        If 
        \begin{align*}
            \eta \geq
            \max\left\{
                2 r \log\left( 2 r \log(2 r) \right),
                \frac{8 (K-1)}{\delta},
                \frac{4}{\delta}
                \log\left(\frac{4 (K-1)}{\delta}\right)
            \right\},
        \end{align*}
        then the transition time $\tau$ to the stable region~\eqref{eq:tau-def} satisfies
        \begin{align*}
            \tau
            \leq T
            \leq \left\lfloor \frac{4 n}{\gamma^2} \right\rfloor + 1.
        \end{align*}

        \item (\textbf{$\varepsilon$-loss at a fixed iteration $T$ independent of $\eta$ and $\varepsilon$.})
        For any $\varepsilon \in (0,1)$, if
        \begin{align*}
            \eta \geq
            \max\left\{
                2 r \log\left( 2 r \log(2 r) \right),
                \frac{8 (K-1)}{\delta},
                \frac{2}{\delta}
                \log\left(\frac{K-1}{\varepsilon}\right)
            \right\},
        \end{align*}
        then $f(\WWeta{T}) < \varepsilon$.
    \end{enumerate}
\end{theorem}

The theorem shows, in particular, that GD can reach $\varepsilon$-loss in a number of iterations that is \emph{independent} of $\varepsilon$ (only the required step size depends on $\varepsilon$).

While the dependence of the step size threshold on $n$ and $\gamma$ is expected \citep{crawshaw2026tight}, a natural question is whether the new dependency on $\delta$ in Theorem~\ref{thm:multiclass-main} is strictly necessary, or if this could be replaced by a uniform threshold depending only on the coarser dataset parameters $n$ and~$\gamma$.
A counterexample shows that this is impossible already in dimension three.
Moreover, the same construction shows that the dependence on $\delta$ in Theorem~\ref{thm:multiclass-main} is essentially sharp.

\begin{theorem} \label{thm:counterexample}
    For every $N \geq 1$, there exists $\eta_N > 0$ such that, for every $\eta \geq \eta_N$, there exists a linearly separable dataset $X_\eta$ with $K=2$ consisting of three points in $\reals^3$, each of norm at most $1$ and with margin at least $1 / \sqrt{26} $, for which GD initialized at zero with step size $\eta$ requires more than $N$ iterations to enter the stable region.
    
    Moreover, let $\delta_\eta$ be the gap~\eqref{eq:delta-def} of $X_\eta$.
    On the family of datasets $X_\eta$:
    \begin{enumerate}
        \item The hard perceptron dynamics~\eqref{eq:multiclass-infinity-system} terminates after $T=1$ iteration for all $\eta$ large enough.
        \item For every fixed $c \in (0, 1)$ and large enough $\eta$, GD initialized at zero with step size $c \, \frac{1}{\delta_\eta} \log\left( \frac{2}{\delta_\eta} \right)$ does not enter the stable region within $T=1$ iteration.
        \item For every $\varepsilon \in (0, \frac{1}{6})$ and large enough $\eta$, GD initialized at zero with any step size smaller than $\frac{1}{\delta_\eta} \log\left( \frac{1}{6 \varepsilon }\right)$ does not achieve loss below $\varepsilon$ at iteration $T=1$.
    \end{enumerate}
\end{theorem}

Stated differently: the conditions on $\eta$ in Theorem~\ref{thm:multiclass-main} must involve some property of the dataset $X$ beyond $\gamma$ and $T$, already with $d=3$, $n=3$ and $K=2$.
We identify $\delta$ as such a property, and the second part of Theorem~\ref{thm:counterexample} shows that if we require GD to enter the stable region by the time the hard perceptron dynamics~\eqref{eq:multiclass-infinity-system} terminates, the dependency on $\delta$ in Theorem~\ref{thm:multiclass-main} is of the correct order.
The proof of Theorem~\ref{thm:counterexample} is postponed to Appendix~\ref{appendix:counterexample}.

\begin{remark}
    The counterexample in Theorem~\ref{thm:counterexample} can be modified so that all datasets in the family have the same maximum margin.
    Let $\gamma_\eta$ be the maximum margin of $X_\eta$, and fix a target margin $\gamma_0 \leq \gamma_\eta$.
    Consider the rescaled dataset $\widetilde{X}_\eta = s_\eta X_\eta$ with $s_\eta = \gamma_0/\gamma_\eta \leq 1$, which has maximum margin $\gamma_0$ and preserves the unit norm bound.
    Under the rescaled step size $\widetilde{\eta} = \eta/s_\eta^2 \geq \eta$, the corresponding GD iterates satisfy $\widetilde{W}^{(t)} = \WWeta{t}/s_\eta$, leaving all class scores and loss values along the trajectory unchanged.
    Since $1 / (2\widetilde{\eta}) = s_\eta^2 / (2\eta) \leq 1 / (2\eta)$, the rescaled trajectory remains outside its stable region for at least as long as the original trajectory, so the same lower bound on the transition time holds.
\end{remark}

\subsection{The hard perceptron dynamics} \label{subsection:hard-perceptron-limit}

In the first part of this section, we show that the hard perceptron dynamics~\eqref{eq:multiclass-infinity-system} is obtained by taking the limit $\eta \to \infty$ in the soft perceptron dynamics~\eqref{eq:multiclass-eta-system}.
\begin{lemma} \label{lemma:eta-infinity-limit}
    Consider the~\eqref{eq:multiclass-eta-system} and ~\eqref{eq:multiclass-infinity-system} iteration maps given by
    \begin{align*}
        G_\eta(Z) = Z + \frac{1}{n} \big(Y - P(\eta Z) \big)^\top X
        \qquad \textrm{and} \qquad 
        G_\infty(Z) = Z + \frac{1}{n} \big( Y - Q(Z) \big)^\top X.
    \end{align*}
    For every $Z \in \reals^{K \times d}$, it holds that $G_\infty(Z) = \lim_{\eta \to \infty} G_\eta(Z) $.
\end{lemma}
\noindent This follows directly from the fact that each row of $P(\eta Z)$ converges as $\eta \to \infty$ to the uniform distribution over the maximizers of each score vector, which is $Q(Z)$.
We include the proof in Appendix~\ref{appendix:shadowing}.

The next lemma shows that the hard perceptron dynamics reaches a perfect classifier in finite time, with an explicit bound on the number of iterations.
This is a full-batch multiclass analogue of the classical perceptron convergence theorem~\citep{novikoff1962convergence}; related multiclass mistake bounds were established by~\citet{crammer2003ultraconservative}.
\begin{lemma} \label{lemma:multiclass-perceptron-finite-time}
    Under~\ref{ass:multiclass1} and~\ref{ass:multiclass2},~\eqref{eq:multiclass-infinity-system} reaches a perfect classifier~\eqref{eq:perfect-classifier-def} in $T \leq \left\lfloor \frac{4 n}{\gamma^2} \right\rfloor + 1$ iterations.
\end{lemma}
\begin{proof}
    Observe that the hard perceptron dynamics~\eqref{eq:multiclass-infinity-system} is an instance of~\eqref{eq:general-system} with $H(Z) = \frac{1}{n} (Y - Q(Z))$, so we can use the results from Section~\ref{section:general-system}.
    For this system, define the \emph{average mistake mass} at iteration $t$ as 
    \begin{align}
        \SSSinf{t} \coloneqq
        S(\ZZinf{t}) = 
        \frac{1}{nK} \sum_{i=1}^n \sum_{k \in [K], k \neq y_i} Q(\ZZinf{t})_{ik}.
    \end{align}
    By Lemma~\ref{lemma:main-perceptron} applied to the hard perceptron dynamics~\eqref{eq:multiclass-infinity-system}, for all $T \geq 1$ we obtain
    \begin{align} \label{eq:perceptron-inequality-for-HP}
        2 K \sum_{t=0}^{T-1} \SSSinf{t} + 2 \sum_{t=0}^{T-1} \langle \ZZinf{t}, H(\ZZinf{t})^\top X \rangle_\mathrm{F} \geq \gamma^2 K^2 \left( \sum_{t=0}^{T-1} \SSSinf{t} \right)^2.
    \end{align}
    We will show that the inner product term is non-positive.
    We have 
    \begin{align} \label{eq:Z_HX_inner}
        \langle \ZZinf{t}, H(\ZZinf{t})^\top X \rangle_\mathrm{F} = 
        \frac{1}{n} \langle \ZZinf{t}, (Y - Q(\ZZinf{t}))^\top X \rangle_\mathrm{F} = 
        \frac{1}{n} \sum_{i=1}^n \sum_{k=1}^K \langle \Zzinf{t}_k, x_i \rangle (Y_{ik} - Q_{ik}),
    \end{align}
    where we use $Q = Q(\ZZinf{t})$ here.
    We argue that the contribution of each data point to the sum is non-positive.
    Recall that $Y_{ik} = 1$ if $k = y_i$ and $Y_{ik} = 0$ otherwise.
    Also, from Eq.~\eqref{eq:multiclass-Q-formula},
    \begin{align} \label{eq:Q-specific}
        Q_{ik} = \frac{1}{\left| \winner{t}_i \right|} \textrm{ if } k \in \winner{t}_i,
        \qquad \textrm{ and } \qquad
        Q_{ik} = 0 \textrm{ otherwise}.
    \end{align}
    By the definition of $\winner{t}_i$, for every data point $x_i$ and every $k \in \winner{t}_i$, it holds that
    \begin{align*}
        \langle \Zzinf{t}_k, x_i \rangle \geq \langle \Zzinf{t}_{y_i}, x_i \rangle.
    \end{align*}
    Accordingly, for every data point $x_i$, it holds that
    \begin{align*}
        \sum_{k=1}^K \langle \Zzinf{t}_k, x_i \rangle (Y_{ik} - Q_{ik})
        &= \langle \Zzinf{t}_{y_i}, x_i \rangle - \sum_{k=1}^K \langle \Zzinf{t}_k, x_i \rangle Q_{ik} \\
        &= \langle \Zzinf{t}_{y_i}, x_i \rangle - \frac{1}{\left| \winner{t}_i \right|} \sum_{k \in \winner{t}_i} \langle \Zzinf{t}_k, x_i \rangle \\
        &= \frac{1}{\left| \winner{t}_i \right|} \sum_{k \in \winner{t}_i}  \langle \Zzinf{t}_{y_i}, x_i \rangle
        - \frac{1}{\left| \winner{t}_i \right|} \sum_{k \in \winner{t}_i} \langle \Zzinf{t}_k, x_i \rangle \\
        &= \frac{1}{\left| \winner{t}_i \right|} \sum_{k \in \winner{t}_i} \langle \Zzinf{t}_{y_i} - \Zzinf{t}_k, x_i \rangle
        \leq 0.
    \end{align*}
    This shows that every term in the sum in Eq.~\eqref{eq:Z_HX_inner} is non-positive, and thus
    \begin{align*}
    \langle \ZZinf{t}, H(\ZZinf{t})^\top X \rangle_\mathrm{F} \leq 0.
    \end{align*}

    Plugging this back in the perceptron-type inequality from Eq.~\eqref{eq:perceptron-inequality-for-HP} and rearranging, we obtain
    \begin{align*}
         \sum_{t=0}^{T-1} \SSSinf{t} \leq \frac{2}{\gamma^2 K},
         \quad \textrm{ for all } T \geq 1.
    \end{align*}

    We claim that if iterate $t$ is not a perfect classifier, then $\SSSinf{t} \geq \frac{1}{2nK}$.
    Combined with the above bound, this shows that the dynamics must terminate in finite time.
    Let $T$ be the index of the first iterate which is a perfect classifier.
    Then, the same bound also gives $\frac{T}{2 n K} \leq \frac{2}{\gamma^2 K}$, and so $T \leq \left\lfloor 4n/\gamma^2 \right\rfloor + 1$, as claimed.

    It remains to show that $\SSSinf{t} \geq \frac{1}{2nK}$ for all iterations $t$ such that there is at least one misclassified point $x_{i_0}$ (that is, $\winner{t}_{i_0} \neq \{ y_{i_0} \}$).
    Then
    \begin{align*}
        \SSSinf{t}
        = \frac{1}{nK} \sum_{i=1}^n \sum_{\substack{k \in [K] \\ k \neq y_i}} Q_{i k} (\ZZinf{t})
        \geq \frac{1}{nK} \sum_{\substack{k \in [K] \\ k \neq y_{i_0}}} Q_{i_0 k} (\ZZinf{t}) 
        = \frac{1}{nK} \left(1 - Q_{i_0 y_{i_0}} (\ZZinf{t}) \right).
    \end{align*}
    If $y_{i_0}\notin\winner{t}_{i_0}$, then $Q_{i_0 y_{i_0}} (\ZZinf{t}) = 0$.
    Otherwise, $\left|\winner{t}_{i_0}\right|\geq 2$, and hence
    \begin{align*}
        Q_{i_0y_{i_0}}(\ZZinf{t}) =
        \frac{1}{\left| \winner{t}_{i_0} \right|} \leq
        \frac{1}{2}.
    \end{align*}
    In either case, we obtain $\SSSinf{t}\geq \frac{1}{2nK}$.
    This concludes the proof.
\end{proof}

\subsection{A shadowing property} \label{section:shadowing-prop}

In the next lemma, we show that when $\eta$ is sufficiently large, the finite-$\eta$ trajectory of~\eqref{eq:multiclass-eta-system} remains close to the limiting trajectory~\eqref{eq:multiclass-infinity-system} up to time $T$.
By definition of the termination time, the~\eqref{eq:multiclass-infinity-system} dynamics reaches a perfect classifier at time $T$, with margin at least $\delta$~\eqref{eq:delta-def}.
The upcoming shadowing bound implies that the~\eqref{eq:multiclass-eta-system} dynamics also classifies every data point correctly at time $T$, with margin at least $\delta/2$.

Recall that the finite-$\eta$ and limiting scores are given by
\begin{align*}
    \Seta{t}_{ik} \coloneqq \langle \Zzeta{t}_k, x_i \rangle,
    \qquad\qquad
    \Sinf{t}_{ik} \coloneqq \langle \Zzinf{t}_k, x_i \rangle.
\end{align*}

\begin{lemma}[Shadowing lemma] \label{lemma:multiclass-shadowing}
    Suppose~\ref{ass:multiclass1} and~\ref{ass:multiclass2} hold.
    Let $T$ be the termination time of the~\eqref{eq:multiclass-infinity-system} dynamics (Lemma~\ref{lemma:multiclass-perceptron-finite-time}), and let $\delta > 0$ be the score gap as defined in Eq.~\eqref{eq:delta-def}.
    If the step size satisfies $\eta \geq \max\left\{ 2r \log\left( 2r \log\left( 2r \right) \right), \frac{8 (K-1)}{\delta} \right\}$, where $r = \max\{1, T/\delta\}$, then
    \begin{align*}
        \langle \Zzeta{T}_{y_i}, x_i \rangle \geq 
        \langle \Zzeta{T}_k, x_i \rangle + \frac{\delta}{2} 
        \qquad \textrm{ for all } i \in [n] \textrm{ and all } k \neq y_i.
    \end{align*}
\end{lemma}

In the proof of the shadowing lemma, we compare the iterates of~\eqref{eq:multiclass-eta-system} and~\eqref{eq:multiclass-infinity-system}.
Their iteration maps only differ through the softmax and hardmax terms: the next lemma bounds this difference.

\begin{lemma} \label{eq:softmax-hardmax-bound}
    Let $\eta > 0$ and $s, \tilde s \in \reals^K$ and consider the vectors
    $p = \softmax(\eta s)$ and $q = \hardmax(\tilde s)$.
    Define the set of indices $\winnernoarg = \argmax{l \in \{1, \dots, K\}} \tilde s_l$.
    Then, 
    \begin{align*}
        \| p - q \| \leq 
        \sqrt{2} \sum_{k \notin \winnernoarg} p_k 
        + \frac{\eta}{2} \| s - \tilde s \|.
    \end{align*}
\end{lemma}
\begin{proof}
    Since $q = \hardmax(\tilde s)$, by definition of $\winnernoarg$ we have
    \begin{align*}
        q_k = 
        \left\{
        \begin{array}{ll}
              \frac{1}{\left|\winnernoarg \right|} & k \in \winnernoarg \\
              0 & k \notin \winnernoarg.
        \end{array} 
        \right.
    \end{align*}
    If $\winnernoarg = [K]$, then $q = \softmax(\eta \tilde s)$, and the result follows directly since $\mathrm{Lip}(\softmax) \leq \frac{1}{2}$ (Lemma~\ref{lemma:softmax-lipschitz}).
    We therefore assume in the remainder that $\winnernoarg \neq [K]$.
    Define $m = \sum_{k \in \winnernoarg } \, p_k$, and also the following probability vectors
    \begin{align*}
        \pi_k = 
        \left\{
        \begin{array}{ll}
              \frac{p_k}{m} & k \in \winnernoarg \\
              0 & k \notin \winnernoarg
        \end{array} 
        \right.
        \qquad\textrm{ and }\qquad 
        \sigma_k = 
        \left\{
        \begin{array}{ll}
              0 & k \in \winnernoarg \\
              \frac{p_k}{1-m} & k \notin \winnernoarg .
        \end{array} 
        \right.
    \end{align*}
    With these definitions, we have $p = m \pi + (1-m) \sigma$, and so
    \begin{align} \label{eq:p-q-bound}
        \| p - q \|
        = \| (1-m) (\sigma-\pi) + \pi - q \| 
        \leq (1-m) \|\sigma-\pi\| + \| \pi - q \| 
        \leq \sqrt{2}(1-m) + \| \pi - q \|,
    \end{align}
    where in the last step we used that $\pi$ and $\sigma$ have disjoint supports, and thus
    \begin{align*}
        \|\sigma-\pi\|^2 =
        \|\sigma\|^2+\|\pi\|^2 \leq
        \|\sigma\|_1^2+\|\pi\|_1^2 = 2.
    \end{align*}

    To bound the term $\| \pi - q \|$, observe that both probability vectors are supported only on $\winnernoarg$.
    For a vector $v \in \reals^K$, let $v_{\winnernoarg} \in \reals^{\left| \winnernoarg \right|}$ be the restriction of $v$ to $\winnernoarg$, indexed by $\winnernoarg$ (that is, $\left( v_{\winnernoarg} \right)_k = v_k$ for $k \in \winnernoarg $).
    With this notation, we have
    \begin{align*}
        \| \pi - q \| = 
        \| \pi_{\winnernoarg} - q_{\winnernoarg} \|.
    \end{align*}
    Finally, observe that for $k \in \winnernoarg$,
    \begin{align*}
        (\pi_{\winnernoarg})_k =
        \frac{p_k}{m} =
        \frac{\softmax(\eta s)_k}{\sum_{l \in \winnernoarg} \softmax(\eta s)_l} = 
        \frac{\exp(\eta s_k)}{\sum_{l \in \winnernoarg} \exp(\eta s_l)} = 
        \softmax(\eta s_{\winnernoarg})_k.
    \end{align*}
    Additionally, letting $\tilde s_{\max} \coloneqq \max_{l \in \{1, \dots, K\}} \tilde s_l$, again for $k \in \winnernoarg$, we also have
    \begin{align*}
        (q_{\winnernoarg})_k = 
        \frac{1}{\big| \winnernoarg \big|} = 
        \frac{\exp(\eta \tilde s_{\max})}{\big| \winnernoarg \big| \exp(\eta \tilde s_{\max})} = 
        \softmax(\eta \tilde s_{\winnernoarg})_k.
    \end{align*}
    Therefore, using $\mathrm{Lip}(\softmax) \leq \frac{1}{2}$ (by Lemma~\ref{lemma:softmax-lipschitz}),
    \begin{align*}
        \| \pi_{\winnernoarg} - q_{\winnernoarg} \|
        = \| \softmax(\eta s_{\winnernoarg}) - \softmax(\eta \tilde s_{\winnernoarg}) \| 
        \leq \eta \; \mathrm{Lip}(\softmax) \; \| s_{\winnernoarg} - \tilde s_{\winnernoarg} \| 
        \leq \frac{\eta}{2} \; \| s - \tilde s \|,
    \end{align*}
    Plugging this back in Eq.~\eqref{eq:p-q-bound}, we obtain
    \begin{align} \label{eq:softmax-hardmax-almost-conclusion}
        \| p - q \| \leq 
        \sqrt{2}(1-m) + \frac{\eta}{2} \| s - \tilde s \| = 
        \sqrt{2} \sum_{k \notin \winnernoarg} p_k + \frac{\eta}{2} \| s - \tilde s \|,
    \end{align}
    as claimed.
\end{proof}

We now turn to the proof of the shadowing lemma.
The argument proceeds by induction over the first $T$ iterations.

\begin{proof}[Proof of Lemma~\ref{lemma:multiclass-shadowing}]
    Define $\Zomega{t} = \| \ZZeta{t} - \ZZinf{t} \|_\mathrm{F}$ and recall that 
    \begin{align*}
        \seta{t}_{ik} = \langle \Zzeta{t}_k, x_i \rangle, 
        \qquad
        \sinf{t}_{ik} = \langle \Zzinf{t}_k, x_i \rangle,
        \qquad \textrm{ and } \qquad 
        \winner{t}_i = \argmax{k \in [K]} \, \sinf{t}_{ik}.
    \end{align*}
    For all $t \in \{0, \dots, T\}$, all $i\in[n]$, all $l \in \winner{t}_i$ and all $k \notin \winner{t}_i$, define 
    \begin{align*}
        \nueta{t}_{ikl} \coloneqq \Seta{t}_{i l} - \Seta{t}_{i k}
        \qquad \textrm{ and } \qquad
        \nuinf{t}_{ikl} = \Sinf{t}_{il} - \Sinf{t}_{ik},
    \end{align*}
    and observe that by the choice of indices, $\nuinf{t}_{ikl} > 0$.
    Now, observe that using~\ref{ass:multiclass1},
    \begin{align} \label{eq:score-diff-bounded-by-omega}
        \| \Seta{t}_i - \Sinf{t}_i \| = 
        \| (\ZZeta{t} - \ZZinf{t}) x_i \| \leq
        \| \ZZeta{t} - \ZZinf{t} \|_\mathrm{F} = \Zomega{t},
    \end{align}
    which in particular implies that $| \Seta{t}_{ik} - \Sinf{t}_{ik} | \leq \Zomega{t}$ for all $k \in [K]$.
    Then, using the inequality $\| a-b \| \leq \|a\| + \|b\| \leq \sqrt{2} \left( \|a\|^2 + \|b\|^2 \right)^{1/2} $, we also have
    \begin{align*}
        \left| \nuinf{t}_{ikl} - \nueta{t}_{ikl} \right|
        &= \left| \Seta{t}_{ik} - \Sinf{t}_{ik} -  (\Seta{t}_{il} - \Sinf{t}_{il}) \right| \\ 
        &= \left| \left \langle (\Zzeta{t}_k - \Zzinf{t}_k) - (\Zzeta{t}_l - \Zzinf{t}_l), x_i \right \rangle \right| \\
        &\leq \left\| \left( \Zzeta{t}_k - \Zzinf{t}_k \right) -  \left( \Zzeta{t}_l - \Zzinf{t}_l \right) \right\| \; \|x_i\| \\
        &\leq \sqrt{2} \left( \| \Zzeta{t}_k - \Zzinf{t}_k \|^2 + \| \Zzeta{t}_l - \Zzinf{t}_l \|^2 \right)^{1/2} \\
        &\leq \sqrt{2} \| \ZZeta{t} - \ZZinf{t} \|_\mathrm{F} 
        = \sqrt{2} \Zomega{t},
    \end{align*}
    where in the last bound we used that $k \neq l$.
    Thus, if we assume that $\Zomega{t} < \frac{\delta}{2\sqrt{2}}$, then it follows that 
    \begin{align*}
        \nuinf{t}_{ikl} - \frac{\delta}{2}
        < \nuinf{t}_{ikl} - \sqrt{2} \Zomega{t}
        \leq \nueta{t}_{ikl} .
    \end{align*}
    Then, since $\nuinf{t}_{ikl} > 0$, from the definition of $\delta$ we have $\nuinf{t}_{ikl} \geq \delta$ and so also $\nueta{t}_{ikl} > \frac{\delta}{2}$.
    This shows that
    \begin{align} \label{eq:small-omega-implication}
        \Zomega{t} < \frac{\delta}{2 \sqrt{2}} 
        \quad\implies\quad
        \nueta{t}_{ikl} > \frac{\delta}{2}
    \end{align}
    In particular, if $\Zomega{T} < \frac{\delta}{2 \sqrt{2}}$, then since $\Sinf{T}_{i y_i} - \Sinf{T}_{ik} \geq \delta$ for all $i \in [n]$ and $k \neq y_i$ (by definition of $\delta$~\eqref{eq:delta-def}), we obtain immediately that $\winner{T}_i = \{ y_i \}$ and hence that $\Seta{T}_{i y_i} - \Seta{T}_{ik} > \frac{\delta}{2}$ for all $k \neq y_i$.
    Therefore, it suffices to show $\Zomega{T} < \frac{\delta}{2\sqrt{2}}$.
    
    \bigskip
    From the definition of the dynamical systems, we have
    \begin{align} \label{eq:multiclass-omega-bound}
        \Zomega{t+1}
        \leq \Zomega{t} + \frac{1}{n} \left\| P(\eta \ZZeta{t}) - Q(\ZZinf{t}) \right\|_\mathrm{F} \, \| X \|_\mathrm{F} 
        \leq \Zomega{t} + \frac{1}{\sqrt{n}} \left\| P(\eta \ZZeta{t}) - Q(\ZZinf{t}) \right\|_\mathrm{F},
    \end{align}
    where in the last step we used~\ref{ass:multiclass1}.

    Let $\pp{t}_i = \softmax(\eta \Seta{t}_i)$ be the $i$-th row of $P(\eta \ZZeta{t})$ and $\qq{t}_i = \hardmax(\Sinf{t}_i)$ be the $i$-th row of $Q(\ZZinf{t})$.
    By Lemma~\ref{eq:softmax-hardmax-bound}, and using Eq.~\eqref{eq:score-diff-bounded-by-omega}, we have
    \begin{align} \label{eq:softmax-hardmax-row-bound}
        \| \pp{t}_i - \qq{t}_i \|_2
        &\leq \sqrt{2} \sum_{k \notin \winner{t}_i} \pp{t}_{ik} + \frac{\eta}{2} \| \Seta{t}_i - \Sinf{t}_i \| \notag \\
        &\leq \sqrt{2} \sum_{k \notin \winner{t}_i} \pp{t}_{ik} + \frac{\eta}{2} \Zomega{t}.
    \end{align}
    Using this, Eq.~\eqref{eq:multiclass-omega-bound} and the condition on the step size, 
    we will show inductively that for all $1 \leq t \leq T$,
    \begin{align} \label{eq:multiclass-induction}
        \Zomega{t} \leq
        \sqrt{2} (K-1) \; \exp\left(-\eta \frac{\delta}{2}\right) \; \sum_{j=0}^{t-1} \left( 1 + \frac{\eta}{2} \right)^j < \frac{\delta}{2\sqrt{2}}.
    \end{align}

    As a first step, observe that if $\Zomega{t} < \frac{\delta}{2\sqrt{2}}$, then from Eq.~\eqref{eq:small-omega-implication} we know that for all $i \in [n]$, any $k \notin \winner{t}_i$ and any $l \in \winner{t}_i$, it holds that
    $\Seta{t}_{il} - \Seta{t}_{ik} > \frac{\delta}{2}$.
    Therefore, as we can always choose some $l \in \winner{t}_i$, for all $k \notin \winner{t}_i$ we have that
    \begin{align} \label{eq:non-winner-prob-small}
        P(\eta \ZZeta{t})_{ik} =
        \frac{\exp(\eta \Seta{t}_{ik})}{\sum_{j=1}^K \exp(\eta \Seta{t}_{ij})} \leq
        \frac{1}{\exp(\eta (\Seta{t}_{il} - \Seta{t}_{ik}))} \leq
        \exp \left( -\eta \frac{\delta}{2} \right).
    \end{align}
    Thus, if $\Zomega{t} < \frac{\delta}{2\sqrt{2}}$, Eq.~\eqref{eq:softmax-hardmax-row-bound} becomes
    \begin{align*}
        \| \pp{t}_i - \qq{t}_i \|_2
        &\leq \sqrt{2} \Big( K - \left| \winner{t}_i \right| \Big) \exp\left(-\eta \frac{\delta}{2} \right) + \frac{\eta}{2} \Zomega{t} \\
        &\leq \sqrt{2} (K-1) \exp\left(-\eta \frac{\delta}{2} \right) + \frac{\eta}{2} \Zomega{t}.
    \end{align*}
    Plugging this back in Eq.~\eqref{eq:multiclass-omega-bound}, we obtain that if $\Zomega{t} < \frac{\delta}{2\sqrt{2}}$, then
    \begin{align} \label{eq:multiclass-omega-bound-refined}
        \Zomega{t+1} \leq
        \left( 1 + \frac{\eta}{2} \right) \Zomega{t} + \sqrt{2} (K-1) \exp\left( -\eta \frac{\delta}{2} \right).
    \end{align}
    We exploit this to show the two inequalities from Eq.~\eqref{eq:multiclass-induction}.

    For the base case $t=1$, since $\Zomega{0} = 0$, Eq.~\eqref{eq:multiclass-omega-bound-refined} gives
    \begin{align*}
        \Zomega{1} \leq \sqrt{2} (K-1) \exp\left(-\eta\frac{\delta}{2} \right),
    \end{align*}
    and this is less than $\frac{\delta}{2\sqrt{2}}$ given the condition on the step size (this will become clear momentarily).

    Now, assume that the conditions from Eq.~\eqref{eq:multiclass-induction} hold at time $t$.
    We show that they also hold at $t+1$.
    Since $\Zomega{t} < \frac{\delta}{2\sqrt{2}}$ by assumption, Eq.~\eqref{eq:multiclass-omega-bound-refined} combined with the assumed bound for $\Zomega{t}$ gives
    \begin{align*}
        \Zomega{t+1}
        &\leq \left( 1 + \frac{\eta}{2} \right) \Zomega{t} + \sqrt{2} (K-1) \exp\left( -\eta \frac{\delta}{2} \right) \\
        &\leq \sqrt{2} (K-1) \; \exp\left(-\eta \frac{\delta}{2}\right) \; \sum_{j=1}^{t} \left( 1 + \frac{\eta}{2} \right)^j + \sqrt{2} (K-1) \exp\left( -\eta \frac{\delta}{2} \right) \\
        &= \sqrt{2} (K-1) \; \exp\left(-\eta \frac{\delta}{2}\right) \; \sum_{j=0}^{t} \left( 1 + \frac{\eta}{2} \right)^j,
    \end{align*}
    which shows the first part.
    For the second part, we need to show that this is less than $\frac{\delta}{2\sqrt{2}}$.
    For this, since we want to show the statement for all $1 \leq t \leq T$, it suffices to show that 
    \begin{align*}
        \sqrt{2} (K-1) \; \exp\left(-\eta \frac{\delta}{2}\right) \; \sum_{j=0}^{T-1} \left(1 + \frac{\eta}{2}\right)^j < \frac{\delta}{2\sqrt{2}}.
    \end{align*}
    Using $\sum_{j=0}^{T-1} (1+a)^j \leq \frac{(1+a)^T}{a}$ for $a > 0$, we have
    \begin{align*}
        \sqrt{2} (K-1) \; \exp\left(-\eta \frac{\delta}{2}\right) \; \sum_{j=0}^{T-1} \left(1 + \frac{\eta}{2} \right)^j &\leq 
        \frac{2\sqrt{2} (K-1)}{\eta} \; \exp\left(-\eta \frac{\delta}{2}\right) \left( 1 + \frac{\eta}{2} \right)^T \\ 
        &= \frac{2\sqrt{2} (K-1)}{\eta} \exp\left(-\eta \frac{\delta}{2} + T \log\left(1 + \frac{\eta}{2}\right) \right)
    \end{align*}
    Using $\eta \geq 2r \log\left( 2r \log\left( 2r \right) \right)$, Lemma~\ref{lemma:eta-condition} implies that $-\eta \frac{\delta}{2} + T \log\left(1 + \frac{\eta}{2}\right) < 0$, so
    \begin{align*}
        \sqrt{2} (K-1) \; \exp\left(-\eta \frac{\delta}{2}\right) \; \sum_{j=0}^{T-1} \left( 1 + \frac{\eta}{2} \right)^j < \frac{2\sqrt{2} (K-1)}{\eta}.
    \end{align*}
    Since $\eta \geq \frac{8 (K-1)}{\delta}$, this quantity is indeed less than $\frac{\delta}{2\sqrt{2}}$, as claimed.
    This concludes the proof.
\end{proof}

\subsection{Transferring the guarantees back to the finite-$\eta$ dynamics}

The shadowing lemma immediately gives an exponential bound on the loss at time $T$.
This is the key step in deriving the main theorem from the finite-time termination of the hard perceptron dynamics.

\begin{proposition} \label{prop:loss-decrease-with-shadowing}
    Under the setting of Lemma~\ref{lemma:multiclass-shadowing}, if GD is initialized at $\WWeta{0} = 0$, then
    \begin{align*}
        f(\WWeta{T}) < (K-1) \, e^{ -\frac{\eta\delta}{2} }.
    \end{align*}
\end{proposition}
\begin{proof}
    For all $i \in [n]$, Lemma~\ref{lemma:multiclass-shadowing} gives
    \begin{align*}
        \langle \Zzeta{T}_{y_i}, x_i \rangle \geq 
        \langle \Zzeta{T}_k, x_i \rangle + \frac{\delta}{2} 
        \quad \textrm{ for all } k \neq y_i.
    \end{align*}
    Going back to the original variable $\Wweta{T}_k = \eta \Zzeta{T}_k$, we obtain
    \begin{align*}
        \langle \Wweta{T}_{y_i} - \Wweta{T}_k , x_i \rangle \geq \frac{\eta \delta}{2}
        \quad \textrm{ for all } k \neq y_i.
    \end{align*}
    This implies that $\exp(\langle \Wweta{T}_{k} - \Wweta{T}_{y_i}, x_i \rangle) \leq \exp\left(-\frac{\eta \delta}{2}\right)$ for all $k \neq y_i$, and so by Eq.~\eqref{eq:multiclass-loss}, the softmax probabilities associated to the correct class satisfy
    \begin{align*}
        [P(\WWeta{T})]_{i y_i}
        = \frac{\exp(\langle \Wweta{T}_{y_i}, x_i \rangle)}{\sum_{j=1}^K \exp(\langle \Wweta{T}_{j}, x_i \rangle)} 
        = \frac{1}{1 + \sum_{j \in [K], j \neq y_i} \exp(\langle \Wweta{T}_{j} - \Wweta{T}_{y_i}, x_i \rangle)} 
        \geq
        \frac{1}{1 + (K-1) e^{ -\frac{\eta \delta}{2} }}.
    \end{align*}
    Since $-\log(\cdot)$ is decreasing, and using $\log(1+x) < x$ for all $x>0$, this implies that 
    \begin{align*}
        f(\WWeta{T}) = 
        -\frac{1}{n} \sum_{i=1}^n \log\left( P(\WWeta{T})_{i y_i} \right) \leq
        \log\left(1 + (K-1) e^{ -\frac{\eta \delta}{2} } \right) < 
        (K-1) e^{-\frac{\eta \delta}{2} },
    \end{align*}
    which is the claimed bound.
\end{proof}

We now use Proposition~\ref{prop:loss-decrease-with-shadowing} to prove Theorem~\ref{thm:multiclass-main}.
Accordingly, choosing $\eta$ so that this bound is less than $1 / (2\eta)$ gives $\tau \leq T$, while choosing it so that the bound is less than $\varepsilon$ gives $f(\WWeta{T}) \leq \varepsilon$.

\begin{proof}[Proof of Theorem~\ref{thm:multiclass-main}, part (i)]
    The first two conditions on the step size ensure that we can apply Proposition~\ref{prop:loss-decrease-with-shadowing}, and we obtain $f(\WWeta{T}) < (K-1) \exp\left(-\frac{\eta \delta}{2} \right)$.
    Then, to show that $\tau \leq T$, it suffices to show that under the third assumption on the step size it holds that $(K-1) \exp\left(-\frac{\eta \delta}{2} \right) < \frac{1}{2\eta}$, which is equivalent to showing $\log\left(2 \eta (K-1) \right) < \frac{\eta \delta}{2}$.    

    Let $x = \frac{\eta \delta}{2}$ and observe that the condition $\eta \geq \frac{4}{\delta} \log\frac{4(K-1)}{\delta}$ implies that $x \geq 2 \log\frac{4(K-1)}{\delta}$.
    Using this together with $\log(\beta) < \frac{\beta}{2}$ for any $\beta > 0$ we obtain 
    \begin{align*}
        \log(2\eta(K-1)) = 
        \log\left(x \frac{4 (K-1)}{\delta}\right) = 
        \log(x) + \log\left( \frac{4 (K-1)}{\delta} \right) < 
        \frac{x}{2} + \frac{x}{2} = x,
    \end{align*}
    which, as argued before, shows that $f(\WWeta{T}) < \frac{1}{2\eta}$.
\end{proof}

\begin{proof}[Proof of Theorem~\ref{thm:multiclass-main}, part (ii)]
    Again, the first two conditions on the step size ensure that $f(\WWeta{T}) < (K-1) \exp\left(-\frac{\eta \delta}{2} \right)$.
    Then, observe that the condition $\eta \geq \frac{2}{\delta} \log\frac{K-1}{\varepsilon}$ 
    implies immediately that 
    $(K-1) \exp\left(-\frac{\eta \delta}{2} \right) \leq \varepsilon$, 
    which shows indeed that $f(\WWeta{T}) < \varepsilon$.
\end{proof}

\section*{Acknowledgements}
We thank Pierre Marion for early discussions about acceleration by large step sizes for binary logistic regression, and Guillaume Wang and  Christopher Criscitiello for further helpful conversations.
AAM thanks Pontus Giselsson and the organizers for the invitation to the ELLIIT Focus Period on Optimization for Learning at Lund University.
AAM is grateful to Michael Crawshaw for insightful discussions on the step-size-independent transition time of binary logistic regression in Lund, which led to the multiclass setting considered here.

We used ChatGPT (August 2026) to sharpen several estimates in our proofs.
The tool also produced the counterexample in Appendix~\ref{appendix:counterexample}, which we subsequently verified, decomposed, and presented ourselves.

\newpage
\bibliographystyle{abbrvnat}
\bibliography{../bibtex/boumal}

\appendix

\section{Equivalence to logistic regression when $K=2$} \label{appendix:K2-binary-equivalence}
Consider a binary classification dataset $\{ (x_i, \widetilde y_i) \}_{i=1}^n$, where $x_i \in \reals^d$ and $\widetilde y_i \in \{-1, 1\}$.
For $\wbin \in \reals^d$, define the binary logistic loss
\begin{align} \label{eq:binary-logistic-loss}
    f_{\mathrm{bin}}(\wbin) \coloneqq&
    \frac{1}{n}
    \sum_{i=1}^n
    \log\left(
        1+\exp\left(-\widetilde y_i\langle \wbin,x_i\rangle\right)
    \right).
\end{align}
The corresponding maximum margin is $\gamma_{\mathrm{bin}} \coloneqq \max_{\|\wbin\|\leq 1} \min_{i\in[n]} \widetilde y_i\langle \wbin,x_i\rangle$.

For the multiclass problem~\eqref{eq:multiclass-loss} with $K=2$, we associate to each label $y_i\in\{1,2\}$ the binary label
$\widetilde y_i=1$ if $y_i=1$ and $\widetilde y_i=-1$ if $y_i=2$.
We denote by $\gamma$ the multiclass maximum margin from Eq.~\eqref{eq:multiclass-max-margin}.

\begin{proposition} \label{prop:binary-multiclass-equivalence}
    Let $w_1, w_2$ be the rows of $W \in \reals^{2 \times d}$.
    GD on the multiclass cross-entropy loss~\eqref{eq:multiclass-loss} with $K=2$ is equivalent to GD on the binary logistic regression loss~\eqref{eq:binary-logistic-loss} under the identifications
    \begin{align*}
        \wbin=w_1-w_2 
        \qquad
        f(W) &=
        f_{\mathrm{bin}}(\wbin),
        \qquad
        \eta_{\mathrm{bin}} =
        2\eta,
        \qquad
        \gamma_{\mathrm{bin}} =
        \frac{\gamma}{\sqrt{2}}.
    \end{align*}
    Moreover, if multiclass GD is initialized at $\WWeta{0}=0$ and we define
    $\wbin_t=\Wweta{t}_1-\Wweta{t}_2$, then
    \begin{align*}
        \Wweta{t}_1 &=
        \frac{\wbin_t}{2},
        \qquad
        \Wweta{t}_2 =
        -\frac{\wbin_t}{2}
    \end{align*}
    for all $t\geq 0$, and $(\wbin_t)_{t\geq 0}$ is the binary logistic GD trajectory with step size $\eta_{\mathrm{bin}}=2\eta$.
\end{proposition}

\begin{proof}
    We first show the equivalence of the losses.
    For $K=2$, we have
    \begin{align*}
        P(W)_{iy_i} =
        \frac{e^{\langle w_{y_i},x_i\rangle}}{e^{\langle w_{y_i},x_i\rangle}+e^{\langle w_{3-y_i},x_i\rangle}} =
        \frac{1}{1+e^{-\langle w_{y_i}-w_{3-y_i},x_i\rangle}} =
        \frac{1}{1+e^{-\widetilde y_i\langle w_1-w_2,x_i\rangle}},
    \end{align*}
    where we used
    $w_{y_i}-w_{3-y_i}=\widetilde y_i(w_1-w_2)$.
    Therefore, using the definition of $f$ in Eq.~\eqref{eq:multiclass-loss}, we obtain
    \begin{align*}
        f(W) =
        \frac{1}{n} \sum_{i=1}^n \log\left( 1+\exp\left(-\widetilde y_i\langle w_1-w_2,x_i\rangle \right) \right) =
        f_{\mathrm{bin}}(w_1-w_2).
    \end{align*}

    We next compare the two maximum margins.
    For any $W \in \reals^{2 \times d}$,
    \begin{align*}
        \|w_1-w_2\|^2 =
        2 \left( \|w_1\|^2 + \|w_2\|^2 \right) - \|w_1 + w_2\|^2 \leq
        2 \| W \|_\mathrm{F}^2.
    \end{align*}
    Thus, $\|W\|_\mathrm{F}\leq 1$ implies $\|w_1-w_2\|\leq\sqrt{2}$.
    Conversely, for any $\wbin\in\reals^d$ such that $\|\wbin\|\leq\sqrt{2}$, the choice
    $w_1=\frac{\wbin}{2}$, $w_2=-\frac{\wbin}{2}$ satisfies
    $w_1-w_2=\wbin$ and $\|W\|_\mathrm{F}^2=\frac{1}{2}\|\wbin\|^2\leq 1$, so
    \begin{align*}
        \left\{ w_1-w_2 \,\middle|\, \|W\|_\mathrm{F}\leq 1 \right\} =
        \left\{ \wbin\in\reals^d \,\middle|\, \|\wbin\|\leq\sqrt{2} \right\}.
    \end{align*}

    For $K=2$, the margin of example $i$ is
    $\widetilde y_i\langle w_1-w_2,x_i\rangle$.
    It follows from the previous identity and the definition of $\gamma_{\mathrm{bin}}$ that
    \begin{align*}
        \gamma =
        \max_{\|\wbin\|\leq\sqrt{2}}
        \min_{i\in[n]} \widetilde y_i\langle \wbin,x_i\rangle =
        \sqrt{2}
        \max_{\|\wbin\|\leq 1}
        \min_{i\in[n]} \widetilde y_i\langle \wbin,x_i\rangle =
        \sqrt{2}\gamma_{\mathrm{bin}}.
    \end{align*}
    Therefore, $\gamma_{\mathrm{bin}}=\frac{\gamma}{\sqrt{2}}$.

    Finally, since $f(W)=f_{\mathrm{bin}}(w_1-w_2)$, the chain rule gives
    \begin{align*}
        \nabla_{w_1} f(W) &=
        \nabla f_{\mathrm{bin}}(w_1 - w_2),
        \qquad
        \nabla_{w_2} f(W) =
        -\nabla f_{\mathrm{bin}}(w_1 - w_2).
    \end{align*}
    Consequently, for $\wbin_t = \Wweta{t}_1 - \Wweta{t}_2$,
    \begin{align*}
        \wbin_{t+1} &=
        \wbin_t-2\eta\nabla f_{\mathrm{bin}}(\wbin_t).
    \end{align*}
    Thus, $(\wbin_t)_{t\geq 0}$ is precisely the binary logistic GD trajectory with
    $\eta_{\mathrm{bin}}=2\eta$.

    Moreover, the two row gradients sum to zero, so
    \begin{align*}
        \Wweta{t+1}_1 + \Wweta{t+1}_2 &=
        \Wweta{t}_1 + \Wweta{t}_2.
    \end{align*}
    Since $\WWeta{0} = 0$, it follows that $\Wweta{t}_1 + \Wweta{t}_2 = 0$ for all $t\geq 0$.
    Combining this with $\wbin_t = \Wweta{t}_1 - \Wweta{t}_2$ gives
    \begin{align*}
        \Wweta{t}_1 &=
        \frac{\wbin_t}{2},
        \qquad
        \Wweta{t}_2 =
        -\frac{\wbin_t}{2},
    \end{align*}
    as claimed.
\end{proof}

\section{Auxiliary results about the multiclass cross-entropy loss} \label{appendix:gd-dynamics-lemma-proof}

\begin{lemma} \label{lemma:softmax-grad}
    For $P$ in~\eqref{eq:multiclass-loss}, it holds that 
    \begin{align}
        \nabla_{w_j} P(W)_{ik} = 
        \big( \mathbbm{1}_{\{ k = j\}} - P(W)_{ij} \big) \, P(W)_{ik} \; x_i.
    \end{align}
\end{lemma}
\begin{proof}
    In the case when $j \neq k$, we have
    \begin{align*}
        \nabla_{w_j} P(W)_{ik} = 
          - \frac{\exp(\langle w_k, x_i \rangle)}{ \big( \sum_{\ell=1}^K \exp(\langle w_\ell, x_i \rangle) \big)^2 } \; \exp(\langle w_j, x_i \rangle) \; x_i = 
          - P(W)_{ik} \; P(W)_{ij} \; x_i.
    \end{align*}
    Otherwise, if $j = k$, we compute
    \begin{align*}
        \nabla_{w_j} P(W)_{ik} = 
        \frac{\exp(\langle w_k, x_i \rangle) \sum_{\ell=1}^K \exp(\langle w_\ell, x_i \rangle) - \exp(\langle w_k, x_i \rangle)^2}{\big( \sum_{\ell=1}^K \exp(\langle w_\ell, x_i \rangle) \big)^2} \; x_i = 
        \big( P(W)_{ik} - P(W)_{ik}^2 \big) \; x_i.
    \end{align*}
    Combining both cases, we obtain the conclusion.
\end{proof}

\begin{lemma} \label{lemma:grad-f}
    The gradient of $f$ in~\eqref{eq:multiclass-loss} is $\nabla f(W) = -\frac{1}{n} \big(Y - P(W)\big)^\top X $.
\end{lemma}
\begin{proof}
    Observe that
    \begin{align} \label{eq:multiclass-grad-start}
        \nabla_{w_j} f(W) = -\frac{1}{n} \sum_{i=1}^n \frac{1}{P(W)_{i y_i}} \nabla_{w_j} P(W)_{i y_i},
    \end{align}
    so we need to compute $\nabla_{w_j} P(W)_{i y_i}$ first.
    Taking $k = y_i$ in Lemma~\ref{lemma:softmax-grad}, we obtain
    \begin{align*}
        \nabla_{w_j} P(W)_{i y_i} = 
        \big( Y_{ij} - P(W)_{ij} \big) \, P(W)_{i y_i} \; x_i.
    \end{align*}
    Plugging this back in Eq.~\eqref{eq:multiclass-grad-start}, we get
    \begin{align*}
        \nabla_{w_j} f(W) = 
        -\frac{1}{n} \sum_{i=1}^n \big( Y_{ij} - P(W)_{ij} \big) \; x_i.
    \end{align*}
    The result follows by stacking these identities for $j \in [K]$.
\end{proof}

\begin{lemma} \label{lemma:softmax-jacobian-bound}
    Let $\Delta^{K-1}$ denote the standard simplex.
    Then, for all $p \in \Delta^{K-1}$, 
    \begin{align} \label{eq:softmax-jacobian-bound}
        0 \preceq
        \diag(p) - p p^\top \preceq
        2 \max_{i \in [K]} p_i (1 - p_i) I.
    \end{align}
\end{lemma}

\begin{proof}
    Observe that the diagonal elements of $\diag(p)-pp^\top$ are $p_i(1-p_i)$, while for each row $i$, the sum of the absolute values of the off-diagonal elements is $ \sum_{j \neq i} p_i p_j = p_i(1-p_i) $.
    By the Gershgorin circle theorem, each eigenvalue of
    $\diag(p)-pp^\top$ lies in at least one interval of the form $ [0, 2 p_i (1 - p_i) ]$.
\end{proof}

We next establish that the Hessian of $f$ is upper bounded in terms of the loss. 
A proof of this bound also appears in~\citep[Lem.~C.3]{emmanouilidis2026sgd}; we include the proof for completeness.

\begin{lemma} \label{lemma:cross-entropy-hessian-self-bound}
    Under~\ref{ass:multiclass1}, the loss function $f$ in Eq.~\eqref{eq:multiclass-loss} satisfies
    \begin{align*}
        \nabla^2 f(W) \preceq 2 f(W) I
    \end{align*}
\end{lemma}
\begin{proof}
    In the proof of Lemma~\ref{lemma:grad-f}, for each $j \in [K]$ we obtained
    \begin{align*}
        \nabla_{w_j} f(W) =
        -\frac{1}{n} \sum_{i=1}^n
        \big( Y_{ij} - P(W)_{ij} \big) \, x_i.
    \end{align*}
    Using Lemma~\ref{lemma:softmax-grad}, it follows that, for every $j,k \in [K]$,
    \begin{align*}
        \nabla_{w_k}\nabla_{w_j} f(W) =
        \frac{1}{n}\sum_{i=1}^n P(W)_{ij} \big( \mathbbm{1}_{\{j=k\}} - P(W)_{ik} \big) \, x_i x_i^\top.
    \end{align*}
    Let $\Delta^{K-1}$ be the standard simplex in $\reals^K$ and $C \colon \Delta^{K-1} \to \reals^{K \times K}$ be given by $C(p) = \diag(p) - p p^\top$.
    Writing $p_i = P(W)_{i:} \in \Delta^{K-1}$, for any $U \in \reals^{K \times d}$ we have
    \begin{align} \label{eq:hess-bound}
        \left\langle U, \nabla^2 f(W)[U] \right\rangle_\mathrm{F} =
        \frac{1}{n}\sum_{i=1}^n (U x_i)^\top C(p_i)(U x_i) \leq
        \frac{2}{n}\sum_{i=1}^n
        \max_{k \in [K]} p_{ik} (1 - p_{ik})
        \| U x_i \|^2,
    \end{align}
    where in the last step we used Lemma~\ref{lemma:softmax-jacobian-bound}.
    We next observe that
    \begin{align*}
        \max_{k \in [K]} p_{ik} (1 - p_{ik}) \leq
        1 - p_{iy_i}.
    \end{align*}
    Indeed, for $k=y_i$ this is immediate, while for $k\neq y_i$,
    \begin{align*}
        p_{ik}(1-p_{ik}) \leq p_{ik}
        \leq 1 - p_{iy_i}.
    \end{align*}
    Using $1-z\leq-\log z$ for every $z\in(0,1]$ and~\ref{ass:multiclass1}, Eq.~\eqref{eq:hess-bound} becomes
    \begin{align*}
        \left\langle U, \nabla^2 f(W)[U] \right\rangle_\mathrm{F}
        \leq \frac{2}{n} \sum_{i=1}^n \big( -\log P(W)_{i y_i} \big) \| U x_i \|^2 
        \leq \frac{2}{n}\sum_{i=1}^n \big( -\log P(W)_{i y_i} \big) \| U \|_\mathrm{F}^2 
        = 2 f(W) \| U \|_\mathrm{F}^2.
    \end{align*}
    Since this holds for every $U \in \reals^{K\times d}$, we conclude that
    $ \nabla^2 f(W) \preceq 2 f(W) I$.
\end{proof}

\section{Additional proofs for Section~\ref{section:two-phase}} \label{appendix:two-phase-proofs}

This section contains the auxiliary results used in Section~\ref{section:two-phase}, as well as the proof of Corollary~\ref{cor:multiclass-accelerated-rate}.

\subsection{Loss decay in the stable region}

We first relate the average mistake mass (Eq.~\eqref{eq:avg-mistake-def}) to the loss in the stable region.
\begin{lemma} \label{lemma:mistake-bound-stable-regime}
    Assume that $f(W) < \frac{1}{2\eta}$ and $\eta \geq \frac{n}{2}$.
    Then, $S(\frac{1}{\eta} W) \geq \frac{f(W)}{2K}$.
\end{lemma}
\begin{proof}
    Apply Lemma~\ref{lemma:mistake-mass} to obtain $ S(\frac{1}{\eta} W) \geq \frac{1}{nK} \left( 1 - \exp(-n f(W))\right)$.
    Then, since $f(W) < \frac{1}{2\eta}$ and $\eta \geq \frac{n}{2}$, we have $f(W) < \frac{1}{n}$.
    Using $1 - \exp(-x) \geq \frac{x}{2}$ for all $x \in [0, 1]$, we obtain
    \begin{align*}
        S \left( \frac{1}{\eta} W \right) \geq
        \frac{1}{nK} \left( 1 - \exp(-n f(W))\right) \geq
        \frac{f(W)}{2K},
    \end{align*}
    as claimed.
\end{proof}

The next result shows that, in the stable region, the gradient norm is bounded below by a constant multiple of the loss.
A binary counterpart (under an alternative assumption) is given in \citep[Lem.~C.5]{crawshaw2026tight}.
\begin{lemma} \label{lemma:gradF-lower-bounded-by-F-if-stable}
    Under Assumption~\ref{ass:multiclass2}, if $f(W) < \frac{1}{2 \eta}$ and $\eta \geq \frac{n}{2}$, then $\| \nabla f(W) \|_\mathrm{F} \geq \frac{\gamma}{2} f(W)$.
\end{lemma}
\begin{proof}
    Recall that $Z = \frac{1}{\eta} W$ and that the soft perceptron dynamics~\eqref{eq:multiclass-eta-system} is a special case of the general batched perceptron dynamics~\eqref{eq:general-system}, with $\Pi(Z) = P(\eta Z)$.
    Hence, Lemma~\ref{lemma:update-term-good-direction-progress} gives
    \begin{align} \label{eq:grad-good-dir-bound}
        \langle W^*, H(Z)^\top X \rangle_\mathrm{F}
        \geq
        \gamma K S \left( \frac{1}{\eta} W \right).
    \end{align}
    By Lemma~\ref{lemma:grad-f},
    \begin{align*}
        -\nabla f(W) =
        \frac{1}{n}\bigl(Y-P(W)\bigr)^\top X =
        H(Z)^\top X,
    \end{align*}
    where we used $W =\eta Z$.
    Therefore, Cauchy--Schwarz on the left-hand side of Eq.~\eqref{eq:grad-good-dir-bound} gives
    \begin{align*}
        \| \nabla f(W) \|_\mathrm{F} \geq
        \langle W^*, -\nabla f(W) \rangle_\mathrm{F} \geq
        \gamma K S \left( \frac{1}{\eta} W \right),
    \end{align*}
    where we also used $\| W^* \|_\mathrm{F} \leq 1$ from Assumption~\ref{ass:multiclass2}.
    
    Then, since $f(W) < \frac{1}{2\eta}$ and $\eta \geq \frac{n}{2}$, conclude with Lemma~\ref{lemma:mistake-bound-stable-regime}, which gives $S(\frac{1}{\eta} W) \geq \frac{f(W)}{2K}$.
\end{proof}

The following lemma shows that the stable region is forward invariant and establishes the descent inequality used in Subsection~\ref{subsection:loss-decay}.
\begin{lemma} \label{lemma:descent-lemma-if-stable}
    Consider a twice continuously differentiable convex $f \colon \reals^{K \times d} \to \reals$ satisfying $\nabla^2 f(W) \preceq 2 f(W) I$ for all $W$.
    Let $\eta > 0$, and define the strict sublevel set $\mathcal S_\eta = \{ W \in \reals^{K \times d} \mid f(W) < \frac{1}{2\eta} \}$.
    If $W \in \mathcal S_\eta$, then 
    \begin{align*}
        f \big(W - \eta \nabla f(W) \big) - f(W) \leq 
        -\frac{\eta}{2} \| \nabla f(W) \|_\mathrm{F}^2.
    \end{align*}
    In particular, this shows that $W \in \mathcal S_\eta$ implies $W^+ \coloneqq W - \eta \nabla f(W) \in \mathcal S_\eta$, so the sublevel set is forward invariant under GD with step size $\eta$.
\end{lemma}
\begin{proof}
    Since $f$ is convex, the sublevel set $\mathcal S_\eta$ is a convex set.
    Furthermore, convexity gives $0 \preceq \nabla^2 f(V)$ for all $V$ and since $\nabla^2 f(V) \preceq 2 f(V) I$ for all $V$, it follows that $\nabla f$ is $1/\eta$-Lipschitz on $\mathcal S_\eta$.

    Now, fix $W \in \mathcal S_{\eta}$, and consider the straight curve $C(t) = W - t \eta \nabla f(W)$ with $C(0) = W \in \mathcal S_\eta$ and $C(1) = W^+$.
    Let $I$ be the connected component of $C^{-1}(\mathcal S_\eta)$ containing $0$.
    Since $\mathcal S_\eta$ is open and $C$ is continuous, every connected component of $C^{-1}(\mathcal S_\eta)$ is an open interval, so $I$ is also an open interval.
    We want to show that $I$ contains $[0, 1]$, so that $W^+ \in \mathcal S_{\eta}$.
    Let $\bar t = \sup I$ and assume for contradiction that $\bar t \leq 1$.
    Since $C(t) \in \mathcal S_\eta$ for all $t \in [0, \bar t)$ and $\nabla f$ is $\frac{1}{\eta}$-Lipschitz on $\mathcal S_{\eta}$, for $t \in [0, \bar t)$ we have that
    \begin{align} \label{eq:descent-in-sublevel-set}
        f(C(t)) \leq
        f(W) - t \eta \| \nabla f(W) \|_\mathrm{F}^2 + \frac{t^2 \eta}{2} \| \nabla f(W) \|_\mathrm{F}^2 =
        f(W) - t \eta \left(1 - \frac{t}{2}\right) \| \nabla f(W) \|_\mathrm{F}^2.
    \end{align}
    Since $0 \leq t < \bar t \leq 1$, this shows that $f(C(t)) \leq f(W) < \frac{1}{2 \eta}$ for all $t \in [0, \bar t)$.
    However, by continuity, it also holds that $f(C(\bar t)) \leq f(W) < \frac{1}{2 \eta}$, which implies that $C(\bar t) \in \mathcal S_\eta$.
    Thus, $\bar t \in I$, and since $I$ is open, it follows that $\sup I > \bar t$, a contradiction.
    Therefore, it must be that $\bar t > 1$, which shows that $W^+ = C(1) \in \mathcal S_\eta$.
    The conclusion then follows from the descent inequality from Eq.~\eqref{eq:descent-in-sublevel-set} at $t=1$.
\end{proof}

\subsection{Proof of Corollary~\ref{cor:multiclass-accelerated-rate}} \label{appendix:corollary-proof}

\begin{proof}
    First, observe that since $T \geq \frac{16 n}{\gamma^2}$, we have $\eta = \frac{\gamma^2 T}{32} \geq \frac{n}{2}$.
    Thus, from the first part of Theorem~\ref{thm:multiclass-two-phase}, we have 
    \begin{align} \label{eq:tau-bound-final}
        \tau \leq 
        \frac{8 \left( \eta + \log\left( 4e(K-1)\eta \right) \right)}{\gamma^2}.
    \end{align}
    We want to show that the right-hand side is less than $\frac{T}{2}$.
    Observe that the second assumption on $T$ implies $\eta = \frac{\gamma^2 T}{32} \geq 2 \log(4e(K-1))$, so
    \begin{align*}
        \log(4e(K-1)\eta) =
        \log(4e(K-1)) + \log(\eta) \leq
        \frac{\eta}{2} + \log(\eta) \leq
        \frac{\eta}{2} + \frac{\eta}{2} = 
        \eta,
    \end{align*}
    where in the last step we used $\log(x) \leq \frac{x}{2}$.
    Therefore, substituting $\eta = \frac{\gamma^2 T}{32}$ in Eq.~\eqref{eq:tau-bound-final} gives
    \begin{align*}
        \tau \leq
        \frac{16 \eta}{\gamma^2} = 
        \frac{T}{2}.
    \end{align*}
    This establishes the first claim.

    We now show the $\mathcal O(1/T^2)$ claim.
    From part (ii) of Theorem~\ref{thm:multiclass-two-phase} applied at $T > \tau$, we have
    \begin{align*}
        f(\WWeta{T}) \leq
        \frac{1}{(T-\tau) \frac{\eta}{8} \gamma^2 + \frac{1}{f(\WWeta{\tau})}} \leq
        \frac{8}{(T-\tau) \eta \gamma^2 }.
    \end{align*}
    Since we showed $\tau \leq \frac{T}{2}$ and using again $\eta = \frac{\gamma^2 T}{32}$, we obtain
    \begin{align*}
        f(\WWeta{T}) \leq
        \frac{16}{T \eta \gamma^2} =
        \frac{512}{\gamma^4 T^2},
    \end{align*}
    as claimed.
\end{proof}

\section{Additional results for Section~\ref{section:independent-step-size}} \label{appendix:shadowing}

\begin{proof}[Proof of Lemma~\ref{lemma:eta-infinity-limit}]
    It suffices to show that $P(\eta Z) \to Q(Z)$ as $\eta \to \infty$.
    For each $i \in [n]$, define the score $s_i \coloneqq Z x_i \in \reals^K$.
    Since $P(\eta Z)$ is a row-softmax matrix (Eq.~\eqref{eq:multiclass-loss}),
    \begin{align} \label{eq:P_eta_Z}
        P(\eta Z)_{ik} = 
        \frac{\exp(\eta \langle z_k, x_i\rangle ) }{\sum_{j=1}^K \exp(\eta \langle z_j, x_i\rangle ) } =
        \frac{\exp(\eta s_{ik}) }{\sum_{j=1}^K \exp(\eta s_{ij} ) } = 
        \frac{1}{\sum_{j=1}^K \exp(\eta (s_{ij} - s_{ik})) }.
    \end{align}
    Observe that
    \begin{align*}
        \lim_{\eta \to \infty} \exp(\eta (s_{ij} - s_{ik}) ) =
        \left\{
        \begin{array}{ll}
              0 & s_{ij} - s_{ik} < 0,\\ 
              1 & s_{ij} - s_{ik} = 0, \\ 
              \infty & s_{ij} - s_{ik} > 0.
        \end{array} 
        \right. 
    \end{align*}
    Therefore, if for fixed $i$ and $k$, there exists an index $j$ such that $s_{ij} - s_{ik} > 0$, the denominator in Eq.~\eqref{eq:P_eta_Z} diverges to $\infty$ and thus
    $\lim_{\eta \to \infty} P(\eta Z)_{ik} = 0$.
    That is, if the score $s_{ik}$ is not maximal among the elements of $s_i$, then the corresponding limit is probability $0$.
    Otherwise, $s_{ik}$ must be the maximal score, and each index $j \in \{1, \dots, K\}$ with equal score $s_{ij} = s_{ik}$ contributes $1$ to the sum in the denominator.
    Therefore, taking the limit $\eta \to \infty$ in Eq.~\eqref{eq:P_eta_Z} we obtain the inverse of the count of maximal scores.
    Combining these two cases, we obtain indeed that $\lim_{\eta \to \infty} P(\eta Z) = Q(Z)$.
\end{proof}

\begin{lemma} \label{lemma:softmax-lipschitz}
    The softmax map is $\frac{1}{2}$-Lipschitz with respect to the Euclidean norm.
    That is,
    \begin{align*}
        \forall v, \tilde v \in \reals^K, \qquad
        \| \softmax(v) - \softmax(\tilde v) \| \leq
        \frac{1}{2}\|v - \tilde v\|.
    \end{align*}
\end{lemma}

\begin{proof}
    Let $P(v) = \softmax(v) \in \reals^K$ be the vector containing the softmax probabilities.
    The gradient of each component is $\nabla_{v_j} P(v)_k = \big(\mathbbm{1}_{\{k=j\}}-P(v)_j\big)P(v)_k$, and so
    \begin{align*}
        \mathrm{Lip}(\softmax)
        = \sup_{v \in \reals^K} \| \diag(P(v))-P(v)P(v)^\top \| 
        \leq \sup_{p \in \Delta^{K-1}} \| \diag(p)-pp^\top \|.
    \end{align*}
    By Lemma~\ref{lemma:softmax-jacobian-bound}, we have $\| \diag(p) - p p^\top \| \leq 2 \max_{i \in [K]} p_i (1-p_i)$.
    However, since $p_i (1 - p_i) \leq \frac{1}{4}$ for $p_i \in [0,1]$, we conclude that $\mathrm{Lip}(\softmax) \leq \frac{1}{2} $.
\end{proof}

\begin{lemma} \label{lemma:eta-condition}
    Let $T, \delta, \eta>0$, and set $r \coloneqq \frac{T}{\delta}$.
    Assume that either $r \leq 1$, or that $r > 1$ and $\eta \geq 2 r \log \!\bigl( 2r \log(2r) \bigr)$.
    Then, it holds that $ -\frac{\eta\delta}{2} + T \log \left(1 + \frac{\eta}{2} \right) < 0$.
\end{lemma}
\begin{proof} [Proof of Lemma~\ref{lemma:eta-condition}]
    Suppose first that $r \leq 1$.
    Since $ \log(1+u) < u$ for every $u>0$, we obtain
    \begin{align*}
        \log\left( 1+\frac{\eta}{2} \right) <
        \frac{\eta}{2} \leq
        \frac{\eta}{2r}.
    \end{align*}
    Using $T=r\delta$, it follows that
    \begin{align*}
        T\log\left( 1 + \frac{\eta}{2} \right) <
        r \delta \frac{\eta}{2r} =
        \frac{\eta\delta}{2}.
    \end{align*}
    This proves the claim when $r\leq 1$.
    Suppose now that $r>1$, and define
    \begin{align*}
        x\coloneqq\frac{\eta}{2r},
        \qquad L\coloneqq\log\!\bigl(2r\log(2r)\bigr),
        \qquad
        q(z) \coloneqq e^z-1-rz.
    \end{align*}
    With this notation, observe that the inequality we need to prove is
    \begin{align*}
        \log(1+ rx) =
        \log\left( 1 + \frac{\eta}{2} \right) <
        \frac{\eta \delta}{2 T} = 
        \frac{\eta}{2 r} = x,
    \end{align*}
    or exponentiating both sides, $1 + rx < e^x$, which is equivalent to $q(x) > 0$.
    We do this by first showing that $q(L) > 0$, then showing that $x \geq L$ and that $q$ is increasing on $[L, \infty)$.
    
    To see that $q(L)>0$, observe that
    \begin{align*}
        q(L) =
        e^L - 1 - r L =
        2r \log(2r) - 1 - r L,
    \end{align*}
    Since $\log(u)\leq u-1$ for $u > 0$, taking $u=\log(2r)$, gives $\log\log(2r)\leq \log(2r)-1 $.
    This implies that
    \begin{align*}
        2 \log(2r) \geq
        \log(2r) + 1 + \log\log(2r) = 
        1 + \log(2r\log(2r)) = 
        1 + L
    \end{align*}
    Since $r>1$, we have
    \begin{align*}
        e^L = 
        2r\log(2r) \geq
        r(1+L) >
        1 + rL,
    \end{align*}
    which shows that $q(L)>0$.
    Additionally, for every $z \geq L$,
    \begin{align*}
        q'(z) =
        e^z-r \geq
        e^L-r =
        r \bigl( 2 \log(2r) - 1 \bigr) > 0,
    \end{align*}
    where the last inequality follows from $r > 1$.
    Thus $q$ is strictly increasing on $[L,\infty)$.
    The assumption on $\eta$ implies $x=\frac{\eta}{2r}\geq L$, and so it follows that $q(x)\geq q(L)>0$.
    As argued before, this implies the claimed inequality.
\end{proof}

\section{A counterexample to uniform transition time bounds in $d=3$} \label{appendix:counterexample}

For the counterexample, it is convenient to work with the binary logistic formulation, since it involves a single parameter vector.
Let $\{(x_i, y_i)\}_{i=1}^n$ be a linearly separable binary classification dataset, with $x_i \in \reals^d$ and $y_i \in \{-1, 1\}$, meaning that there exist $w^* \in \reals^d$ and $\gamma > 0$ such that
\begin{align*}
    \| w^* \| \leq 1,
    \qquad
    y_i \langle w^*, x_i\rangle \geq \gamma
    \quad \text{ for all } i \in [n].
\end{align*}
For $w \in \reals^d$, we consider the \emph{binary logistic loss}
\begin{align}
    f_{\mathrm{bin}}(w)
    \coloneqq
    \frac{1}{n}\sum_{i=1}^n \ell\bigl(y_i \langle w,x_i\rangle\bigr),
    \qquad
    \ell(z) \coloneqq \log(1+\exp(-z)).
\end{align}

\noindent We also consider the corresponding \emph{binary perceptron dynamics}, given by
\begin{align} \label{eq:binary-perceptron}
    z_{t+1}^\infty =
    z_t^\infty + \frac{1}{n}\sum_{i=1}^n \varphi \left( \langle z_t^\infty, y_i x_i\rangle \right) y_i x_i,
    \qquad
    z_0^\infty=0,
\end{align}
where $\varphi(s) = 1$ for $s < 0$, $\varphi(0) = 1/2$, and $\varphi(s) = 0$ for $s > 0$, along with its \emph{binary score gap}
\begin{align} \label{eq:binary-gap}
    \delta_\mathrm{bin}
    \coloneqq
    \min\left\{
        \left| y_i \langle z_t^\infty, x_i \rangle \right| \mid
        0 \leq t \leq T,\ i \in [n], \
        y_i \langle z_t^\infty, x_i \rangle \neq 0
    \right\}.
\end{align}

When $K=2$, the multiclass formulation is equivalent to binary logistic regression under the identification $w=w_1-w_2$, with $f(W)=f_{\mathrm{bin}}(w)$ and $\eta_{\mathrm{bin}}=2\eta$ (see Appendix~\ref{appendix:K2-binary-equivalence}), so the stable-region condition $f(W)<\frac{1}{2\eta}$ becomes $f_{\mathrm{bin}}(w)<\frac{1}{\eta_{\mathrm{bin}}}$.
Similarly, the rows of the $K=2$ hard perceptron iterate~\eqref{eq:multiclass-infinity-system} at time $t$ are $z_t^\infty$ and $-z_t^\infty$, and the multiclass gap~\eqref{eq:delta-def} satisfies $\delta = 2\delta_{\mathrm{bin}}$. 
(In particular, the two perceptron dynamics have the same termination time).
Writing $\eta$ for $\eta_{\mathrm{bin}}$ below for notational convenience, to prove Theorem~\ref{thm:counterexample} it suffices to show the following result for binary logistic regression.

\begin{theorem} \label{thm:binary-counterexample}
    For every $N \geq 1$, there exists $\eta_N > 0$ such that, for every $\eta \geq \eta_N$, there exists a linearly separable dataset $X_\eta$ of three points in $\reals^3$, each of norm at most $1$ and with margin at least $1 / (2\sqrt{13})$, for which GD initialized at zero with step size $\eta$ requires more than $N$ iterations to enter the stable region.
    
    Moreover, let $\delta_\eta \coloneqq \delta_\mathrm{bin}(X_\eta)$ be the binary gap~\eqref{eq:binary-gap} of $X_\eta$.
    On the family of datasets $X_\eta$:
    \begin{enumerate}[label=(C\arabic*)]
        \item \label{thm:hp-transition-1} The hard binary perceptron dynamics~\eqref{eq:binary-perceptron} terminates after $T=1$ iteration if $\eta$ is large enough.
        \item \label{thm:delta-dep-transition} For every fixed $c \in (0, 1)$ and large enough $\eta$, GD initialized at zero with step size $c \, \frac{1}{\delta_\eta} \log\left( \frac{1}{\delta_\eta} \right)$ does not enter the stable region within $T=1$ iteration.
        \item \label{thm:delta-dep-eps-loss} For every $\varepsilon \in (0, \frac{1}{6})$ and large enough $\eta$, GD initialized at zero with step size smaller than $\frac{1}{\delta_\eta} \log\left( \frac{1}{6 \varepsilon }\right)$ does not achieve loss below $\varepsilon$ at iteration $T=1$.
    \end{enumerate}
     
\end{theorem}

The counterexample is obtained in two stages.
We first construct a two-dimensional example with two points for which, from a suitable $\eta$-dependent initialization, the rescaled GD dynamics have a large-$\eta$ limit whose rescaled loss remains bounded away from zero.
We then embed the construction in three dimensions and add a third, $\eta$-dependent point so that one GD step from $\weta{0}=0$ produces the required initialization in the first two coordinates.
The third point is chosen to be well classified after this step, so that its subsequent contribution vanishes as $\eta \to \infty$.
Then, from the second iterate onward, the three-dimensional dynamics recover the limiting behavior of the two-dimensional construction.

Let $q(z) \coloneqq -\ell^\prime(z) = \frac{1}{1 + \exp(z)}$.
GD with step size $\eta > 0$ then has the update rule
\begin{align} \label{eq:binary-gd}
    w_{t+1} =
    w_t + \frac{\eta}{n} \sum_{i=1}^n q \bigl(y_i \langle w_t, x_i\rangle\bigr) \, y_i x_i.
\end{align}

We prove Theorem~\ref{thm:binary-counterexample} through the sequence of auxiliary lemmas below.

\subsection{An auxiliary two-dimensional construction with $\eta$-dependent initialization}

We begin with the two-dimensional dataset consisting of the points
\begin{align} \label{eq:2d-dataset}
    \xp = (a, b),
    \qquad
    \xm = (a, -b),
\end{align}
both with label $y_+ = y_- = 1$, where suitable values for the parameters $a, b > 0$ will be chosen later.
The scores of these two points are 
\begin{align*}
    \langle \weta{t}, \xp \rangle =
    a \weta{t, 1} + b \weta{t, 2},
    \qquad \textrm{ and } \qquad
    \langle \weta{t}, \xm \rangle =
    a \weta{t, 1} - b \weta{t, 2}.
\end{align*}
We define the alternative coordinates
\begin{align} \label{eq:xi-vt-def}
    \xieta{t} = a \weta{t, 1} - \log \eta,
    \qquad \textrm{ and } \qquad
    \veta{t} = b \weta{t, 2}.
\end{align}
In these coordinates, the scores become
\begin{align} \label{eq:margins-wrt-xi-v}
    \langle \weta{t}, \xp \rangle = \log \eta + \xieta{t} + \veta{t},
    \qquad \textrm{ and } \qquad
    \langle \weta{t}, \xm \rangle = \log \eta + \xieta{t} - \veta{t}.
\end{align}
For consistency with the construction with three points in $\reals^3$ later, we use the $1/3$ normalization for this auxiliary two-dimensional dynamics.
Thus, we consider the auxiliary update
\begin{align} \label{eq:two-point-gd-xi-v}
    \weta{t+1} =
    \weta{t} + \frac{\eta}{3} \left( \qetap{t} \, \xp + \qetam{t} \, \xm \right),
\end{align}
where we defined $\qetap{t}$ and $\qetam{t}$ as the logistic weights for $\xp$ and $\xm$, that is,
\begin{align*}
    \qetap{t} \coloneqq q(\log \eta + \xieta{t} + \veta{t}),
    \qquad \textrm{ and } \qquad
    \qetam{t} \coloneqq q(\log \eta + \xieta{t} - \veta{t}).
\end{align*}

Since we are interested in keeping the loss above the $1/\eta$ scale, it will be useful to also track the rescaled loss
\begin{align} \label{eq:rhoeta-def}
    \rhoeta{t} \coloneqq
    \frac{\eta}{3} \Big( \ell(\log \eta + \xieta{t} + \veta{t}) + \ell(\log \eta + \xieta{t} - \veta{t}) \Big).
\end{align}

In the counterexample, we want to keep the loss value above the $1/\eta$ threshold for a certain number of iterations.
The following lemma shows that if a score is roughly $\log \eta$, then its contribution to the loss is on the order of $1/\eta$.
By contrast, the contribution of points with scores much larger than $\log \eta$ vanishes asymptotically.
\begin{lemma} \label{lemma:eta-scale}
    Let $\meta = \log \eta + \reta$.
    Then, as $\eta \to \infty$:
    \begin{enumerate}
        \item If $\reta \to \infty$, then $\eta \ell(\meta) \to 0$.
        \item If $\reta \to c$, then $\eta \ell(\meta) \to e^{-c}$.
    \end{enumerate}
\end{lemma}
\begin{proof}
    Since $\log(1 + z) \leq z$ for all $z \geq 0$, we have $\ell(z) \leq \exp(-z) $.
    At the same time, the logistic loss has the self-boundedness property $q(z) \leq \ell(z)$.
    This holds since $\frac{z}{1+z} \leq \log(1+z)$ for all $z \geq 0$.
    Thus,
    \begin{align*}
        \frac{1}{\frac{1}{\eta} + \exp(\reta)} = 
        \frac{\eta}{1 + \eta \exp(\reta)} =
        \eta q(\meta) \leq 
        \eta \ell(\meta) \leq
        \eta \exp(-\meta) = 
        \exp(-\reta).
    \end{align*}
    From here, the conclusion is clear by taking the corresponding limit on both sides.
\end{proof}

We first write the finite-$\eta$ dynamics explicitly in the $(\xi,v)$ coordinates.
\begin{lemma} \label{lemma:xi-v-dynamics}
    The dynamics of $(\xieta{t},\veta{t})$ is given by
    \begin{align*}
        (\xieta{t+1}, \veta{t+1}) =
        \varPhi_\eta (\xieta{t}, \veta{t}),
        \quad\quad
        \textrm{ where }
        \varPhi_\eta(\xi, v) \coloneqq
        \begin{pmatrix}
            \xi + \dfrac{a^2}{3}
            \left( \dfrac{1}{\eta^{-1} + e^{\xi + v}} + \dfrac{1}{\eta^{-1} + e^{\xi - v}} \right) \\[1.2em]
            v + \dfrac{b^2}{3} \left( \dfrac{1}{\eta^{-1} + e^{\xi + v}} - \dfrac{1}{\eta^{-1} + e^{\xi - v}} \right)
        \end{pmatrix}.
    \end{align*}
\end{lemma}

\begin{proof}
    Using the definitions of $\xieta{t}$ and $\veta{t}$ together with the GD update in Eq.~\eqref{eq:two-point-gd-xi-v}, we have
    \begin{align*}
        \xieta{t+1} =
        a \weta{t+1, 1} - \log \eta =
        \xieta{t} + \frac{a^2 \eta}{3} \left( \qetap{t} + \qetam{t} \right),
    \end{align*}
    and similarly for $\veta{t}$, we obtain
    \begin{align*}
        \veta{t+1} =
        b \weta{t+1, 2} =
        \veta{t} + \frac{b^2 \eta}{3} \left( \qetap{t} - \qetam{t} \right).
    \end{align*}
    From the expressions for the scores in Eq.~\eqref{eq:margins-wrt-xi-v},
    \begin{align*}
        \eta \qetap{t} &=
        \frac{\eta}{1 + \eta e^{\xieta{t} + \veta{t}}} =
        \frac{1}{\eta^{-1} + e^{\xieta{t} + \veta{t}}}, \\
        \eta \qetam{t} &=
        \frac{\eta}{1 + \eta e^{\xieta{t} - \veta{t}}} =
        \frac{1}{\eta^{-1} + e^{\xieta{t} - \veta{t}}}.
    \end{align*}
    Substituting these expressions into the two update rules gives the claimed iteration map.
\end{proof}

The iteration map has a well-defined limit as $\eta \to \infty$.
Indeed, for every fixed $(\xi,v) \in \reals^2$,
\begin{align*}
    \varPhi_\eta(\xi,v) \to
    \tilde\varPhi(\xi,v) \coloneqq
    \begin{pmatrix}
        \xi + \frac{2 a^2}{3} e^{-\xi} \cosh(v) \\
        v - \frac{2 b^2}{3} e^{-\xi} \sinh(v)
    \end{pmatrix}.
\end{align*}
Let $(\xiinf{t},\vinf{t})$ denote the iterates of the limiting map $\tilde\varPhi$, initialized at the same $\eta$-independent point as the finite-$\eta$ dynamics, namely
\begin{align} \label{eq:xi-v-limiting-system}
    (\xiinf{t+1}, \vinf{t+1}) =
    \tilde\varPhi (\xiinf{t}, \vinf{t}),
    \qquad\qquad
    (\xiinf{0}, \vinf{0}) = (\xi_0, v_0).
\end{align}

Then, we obtain convergence of the corresponding trajectories over any fixed number of steps.
\begin{lemma} \label{lemma:2d-eta-system-convergence}
    For every fixed $t \geq 0$, it holds that $\lim_{\eta \to \infty} (\xieta{t},\veta{t}) = (\xiinf{t},\vinf{t})$.
\end{lemma}

\begin{proof}
    We argue by induction on $t$.
    The claim is immediate at $t=0$ since the two systems start from the same initial point.
    Now, fix $t \geq 0$ and suppose that
    $(\xieta{t}, \veta{t}) \to (\xiinf{t},\vinf{t})$.
    By the explicit form of $\varPhi_\eta$, it follows that
    $\varPhi_\eta (\xieta{t}, \veta{t}) \to \tilde \varPhi(\xiinf{t}, \vinf{t})$.
    Therefore, $ (\xieta{t+1}, \veta{t+1}) \to (\xiinf{t+1}, \vinf{t+1}) $, which completes the induction.
\end{proof}

We next characterize the limiting behavior of the rescaled loss and derive the corresponding dynamics in the $(\rho, v)$ variables.
\begin{lemma} \label{lemma:rho-limit-and-system}
    For every fixed $t$, the rescaled loss $\rhoeta{t}$ (Eq.~\eqref{eq:rhoeta-def}) has a limit as $\eta \to \infty$:
    \begin{align*}
        \rhoeta{t} \to
        \rhoinf{t} \coloneqq \frac{2}{3} e^{-\xiinf{t}} \cosh(\vinf{t}).
    \end{align*}
    Moreover, the limiting variables $(\rhoinf{t}, \vinf{t})$ satisfy
    \begin{align} \label{eq:rho-v-limiting-system}
        \rhoinf{t+1} &= \rhoinf{t} e^{-a^2 \rhoinf{t}} \frac{\cosh(\vinf{t+1})}{\cosh(\vinf{t})}, \notag \\
        \vinf{t+1} &= \vinf{t} - b^2 \rhoinf{t} \tanh(\vinf{t}).
    \end{align}
\end{lemma}

\begin{proof}
    By Lemma~\ref{lemma:2d-eta-system-convergence}, $\xieta{t} \to \xiinf{t}$ and $\veta{t} \to \vinf{t}$ as $\eta \to \infty$.
    Lemma~\ref{lemma:eta-scale} therefore gives
    \begin{align*}
        3 \rhoeta{t}
        = \eta \Big( \ell(\log \eta + \xieta{t} + \veta{t}) + \ell(\log \eta + \xieta{t} - \veta{t}) \Big)
        \to e^{-\xiinf{t} - \vinf{t}} + e^{-\xiinf{t} + \vinf{t}}
        = 2 e^{-\xiinf{t}} \cosh(\vinf{t}),
    \end{align*}
    which proves the first claim.
    For the second claim, using the definition of $\rhoinf{t}$, the limiting dynamics in Eq.~\eqref{eq:xi-v-limiting-system} can be written as
    \begin{align*}
        \xiinf{t+1} = \xiinf{t} + a^2 \rhoinf{t},
        \qquad \textrm{ and } \qquad
        \vinf{t+1} = \vinf{t} - b^2 \rhoinf{t} \tanh(\vinf{t}).
    \end{align*}
    Observe that the second one is exactly the update rule of $\vinf{t}$ from Eq.~\eqref{eq:rho-v-limiting-system}.
    Also,
    \begin{align*}
        \rhoinf{t+1}
        = \frac{2}{3} e^{-\xiinf{t+1}} \cosh(\vinf{t+1})
        = \frac{2}{3} e^{-\xiinf{t}} e^{-a^2 \rhoinf{t}} \cosh(\vinf{t+1})
        = \rhoinf{t} e^{-a^2 \rhoinf{t}} \frac{\cosh(\vinf{t+1})}{\cosh(\vinf{t})},
    \end{align*}
    which gives Eq.~\eqref{eq:rho-v-limiting-system}.
\end{proof}

We next characterize a forward invariant region for the $(\rhoinf{t}, \vinf{t})$ dynamics.
\begin{lemma} \label{lemma:invariant-region}
    If $b > a$, there exist $\tilde v^* > 0, c > 0$ and $R > 0$ such that the region 
    \begin{align*}
        \mathcal R = \{ (\tilde\rho, \tilde v) \, \mid \, \tilde \rho \geq R, \, \tilde v^* \leq \left| \tilde v \right| \leq c \tilde\rho \}
    \end{align*}
    is invariant under the limiting dynamical system in Eq.~\eqref{eq:rho-v-limiting-system}.
    Specifically, it suffices to pick $\tilde v^*$ large enough such that $\tanh(\tilde v^*) > a^2 / b^2$, take $c = (b^2 \tanh(\tilde v^*) - a^2) / 4$ and $R > 0$ large enough such that 
    \begin{align*}
        R (b^2 \tanh(\tilde v^*) - c ) > \tilde v^*,
        \quad \textrm{ and } \quad
        \frac{1}{2} \exp \left( (-a^2 + b^2 \tanh(\tilde v^*) - 2c) R \right) > \max \left\{1, \frac{b^2}{c} \right\}.
    \end{align*}
\end{lemma}
\begin{proof}
    Let $(\tilde \rho, \tilde v) \coloneqq (\rhoinf{t}, \vinf{t}) \in \mathcal R$.
    We want to show that $(\tilde \rho_+, \tilde v_+) \coloneqq (\rhoinf{t+1}, \vinf{t+1}) \in \mathcal R$, too.
    Let $\tilde s = \left| \tilde v \right|$ and $\tilde s_+ = \left| \tilde v_+ \right|$, and observe that since $\cosh(\cdot)$ is even and $\tanh(\cdot)$ is odd, the dynamics becomes
    \begin{align*}
        \tilde \rho_+ = \tilde \rho e^{-a^2 \tilde \rho} \frac{\cosh(\tilde s_+)}{\cosh(\tilde s)},
        \qquad \textrm{ and } \qquad 
        \tilde s_+ = \left| \tilde s - b^2 \tilde \rho \tanh(\tilde s) \right|.
    \end{align*}

    We claim that by the choice of $c$, it holds that $b^2 \tanh(\tilde v^*) > c$.
    Indeed,
    \begin{align*}
        b^2 \tanh(\tilde v^*) - c =
        \frac{3 b^2 \tanh(\tilde v^*) + a^2}{4} > 0.
    \end{align*}
    Since $(\tilde \rho, \tilde v)$ belongs to $\mathcal R$, we have $\tilde s \geq \tilde v^*$ and so $b^2 \tanh(\tilde s) \geq b^2 \tanh(\tilde v^*) > c$ by the previous inequality.
    This shows that $b^2 \tilde\rho \tanh(\tilde s) > c \tilde \rho \geq \left| \tilde v \right| = \tilde s$ (again since the initial point belongs to $\mathcal R$), and so
    \begin{align} \label{eq:tilde_s_plus}
        \tilde s_+ = b^2 \tilde \rho \tanh(\tilde s) - \tilde s .
    \end{align}
    Using again $\tilde s \in [\tilde v^*, c \tilde \rho]$, this gives
    \begin{align*}
        b^2 \tilde \rho \tanh(\tilde v^*) - c \tilde \rho \leq
        b^2 \tilde \rho \tanh(\tilde v^*) - \tilde s \leq
        \tilde s_+ \leq
        b^2 \tilde \rho \tanh(c \tilde \rho) - \tilde s \leq 
        b^2 \tilde \rho \tanh(c \tilde \rho) \leq 
        b^2 \tilde \rho
    \end{align*}
    Therefore, since $\tilde \rho \geq R$, in order to show $\tilde s_+ \in [\tilde v^*, c \tilde \rho_+]$, it suffices to impose that
    \begin{align} \label{eq:sufficient-condition-for-tilde-s}
        R (b^2 \tanh(\tilde v^*) - c) \geq \tilde v^*,
        \qquad \textrm{ and } \qquad
        b^2 \tilde\rho \leq c \tilde \rho_+.
    \end{align}
    We already require $R$ large enough so that the first condition is satisfied.
    We show that the conditions we imposed also imply that the second condition is satisfied.
    Using $\frac{1}{2} e^x \leq \cosh(x) \leq e^x$ for all $x > 0$, we obtain
    \begin{align} \label{eq:rho_plus_rho_ratio}
        \frac{\tilde \rho_+}{\tilde \rho} = 
        e^{-a^2 \tilde \rho} \frac{\cosh(\tilde s_+)}{\cosh(\tilde s)} \geq
        \frac{1}{2} e^{-a^2 \tilde \rho + \tilde s_+ - \tilde s} =
        \frac{1}{2} e^{-a^2 \tilde \rho + b^2 \tilde \rho \tanh(\tilde s) - 2 \tilde s} \geq
        \frac{1}{2} e^{(-a^2 + b^2 \tanh(\tilde v^*) - 2 c) R},
    \end{align}
    where the third step follows from expression for $\tilde s_+$ from Eq.~\eqref{eq:tilde_s_plus}, and the last step uses $\tilde s \in [\tilde v^*, c \tilde \rho]$ and $\tilde \rho \geq R$.
    By the choices of $c$ and $\tilde v^*$, we have
    \begin{align*}
        -a^2 + b^2 \tanh(\tilde v^*) - 2c = 
        \frac{ b^2 \tanh(\tilde v^*) - a^2}{2} > 0,
    \end{align*}
    so we can make the right-hand side of Eq.~\eqref{eq:rho_plus_rho_ratio} larger by taking a larger $R$.
    Given the second condition we impose on $R$, we obtain both $\frac{\tilde \rho_+}{\tilde \rho} \geq \frac{b^2}{c}$, which by Eq.~\eqref{eq:sufficient-condition-for-tilde-s} suffices to obtain $\tilde s_+ \in [\tilde v^*, c \tilde \rho_+]$, as well as $\tilde \rho_+ \geq \tilde \rho \geq R$, which is the only condition that $\tilde \rho_+$ must satisfy to conclude that $(\tilde \rho_+, \tilde s_+) \in \mathcal R$.
\end{proof}

We now choose concrete values of $a$ and $b$ and initialize the limiting dynamics inside the invariant region.
\begin{lemma} \label{lemma:invariant-region-specialized}
    Take $a = \frac{1}{4}$ and $b = \frac{1}{2}$, and initialize Eq.~\eqref{eq:rho-v-limiting-system} at $(\rhoinf{0},\vinf{0}) = (80,\log(2))$.
    Then,
    \begin{align*}
        (\rhoinf{t}, \vinf{t}) \in \mathcal R \coloneqq  \left\{ (\tilde \rho, \tilde v) \mid \tilde \rho \geq 72, \, \log(2) \leq \left| \tilde v \right| \leq \frac{7}{320} \tilde \rho \right\}
        \qquad \textrm{ for all } t \geq 0.
    \end{align*}
\end{lemma}
\begin{proof}
    First, we use the conditions in Lemma~\ref{lemma:invariant-region} to pick the constants for defining the invariant region $\mathcal R$ for the specific choices of $b$ and $a$.
    Since $a^2 / b^2 = 1/4$, it suffices to have $\tilde v^* > \tanh^{-1}(\frac{1}{4})$, so we choose $\tilde v^\ast = \log(2)$.
    This gives $c = (b^2 \tanh(\tilde v^*) - a^2) / 4 = 7/320$.
    It remains to pick $R > 0$ so that 
    \begin{align*}
        R > \frac{\tilde v^*}{b^2 \tanh(\tilde v^*) - c} \approx 5.4,
        \quad \textrm{ and } \quad
        R > \frac{\log\left( 2 \max \left\{ 1, \frac{b^2}{c} \right\} \right)}{-a^2 + b^2 \tanh(\tilde v^*) - 2 c} \approx 71.5,
    \end{align*}
    so we choose $R = 72$.
    This gives the region $\mathcal R = \left\{ (\tilde \rho, \tilde v) \mid \tilde\rho \geq 72, \, \left| \tilde v \right| \in [\log(2), \frac{7}{320} \tilde \rho] \right\}$.
    Since $(\rhoinf{0}, \vinf{0}) = (80,\log(2)) \in \mathcal R$, Lemma~\ref{lemma:invariant-region} implies that $(\rhoinf{t}, \vinf{t}) \in \mathcal R$ for all $t \geq 0$.
\end{proof}

We now choose the initialization of the finite-$\eta$ dynamics so that the corresponding limiting dynamics in Eq.~\eqref{eq:rho-v-limiting-system} starts from $(\rhoinf{0}, \vinf{0}) = (80, \log(2))$.
Since $\rhoinf{0} = \frac{2}{3} e^{-\xiinf{0}} \cosh(\vinf{0})$ (by Lemma~\ref{lemma:rho-limit-and-system}), we obtain $\xiinf{0} = -\log(96)$.
Thus, using the definitions in Eq.~\eqref{eq:xi-vt-def}, we choose
\begin{align*}
    \weta{0} = \left(4 \log\frac{\eta}{96}, 2\log(2) \right),
\end{align*}
so that $(\xieta{0}, \veta{0}) = (\xiinf{0}, \vinf{0})$.

\subsection{A three-dimensional construction from zero initialization}

With $a = \frac{1}{4}, b = \frac{1}{2}$, consider now the dataset consisting of three points
\begin{align} \label{eq:3d-dataset}
    x_1 = (a, b, 0), \qquad
    x_2 = (a, -b, 0), \qquad
    x_{3, \eta} = \left( -\frac{1}{2} + \frac{24}{\eta} \log\frac{\eta}{96}, \frac{12 \log 2}{\eta}, \frac{1}{2} \right),  
\end{align}
with $y_1 = y_2 = y_3 = 1$.
Just as in the two-dimensional construction, define
\begin{align*}
    \xieta{t} = \frac{1}{4} \weta{t, 1} - \log\eta,
    \qquad \textrm{ and } \qquad
    \veta{t} = \frac{1}{2} \weta{t, 2},
\end{align*}
and the rescaled loss $ \rhoeta{t} = \eta f_\mathrm{bin}(\weta{t}) $.
We first identify the limiting behavior of $(\xieta{t}, \veta{t})$ as $\eta \to \infty$, which matches the limiting two-dimensional dynamics $(\xiinf{t}, \vinf{t})$ (with a one-step delay).
We also show that the score of the third point remains of order $\eta$, so that its contribution to both the GD update and the rescaled loss vanishes as $\eta \to \infty$.
\begin{lemma} \label{lemma:3d-induction}
    Let $\weta{t}$ denote the GD iterates on the dataset in Eq.~\eqref{eq:3d-dataset}, starting from $\weta{0} = 0$.
    Let $(\xiinf{t}, \vinf{t})$ denote the limiting system from the two-dimensional construction (Eq.~\eqref{eq:xi-v-limiting-system}), initialized at
    \begin{align*}
        (\xiinf{0}, \vinf{0}) = (-\log(96), \log(2)).
    \end{align*}
    Then, for every fixed $t \geq 1$,
    \begin{align} \label{eq:xi-v-induction}
        \lim_{\eta \to \infty} (\xieta{t}, \veta{t}) = (\xiinf{t-1}, \vinf{t-1}),
        \quad\quad \textrm{ and } \quad\quad
        \langle \weta{t}, x_{3, \eta} \rangle = \frac{\eta}{24} + \mathcal O(\log \eta).
    \end{align}
\end{lemma}

\begin{proof}
    Since $\weta{0} = 0$, we have $q(\langle \weta{0}, x_i \rangle) = \frac{1}{2}$ for all $i \in \{1, 2, 3\}$ and using the GD update rule $\weta{t+1} = \weta{t} + \frac{\eta}{3} \sum_{i=1}^3 q(\langle \weta{t}, x_i \rangle) \; x_i$, we immediately have
    \begin{align*}
        \weta{1} = 
        \frac{\eta}{6} (x_1 + x_2 + x_{3, \eta}) = 
        \left( 4 \log\frac{\eta}{96}, 2 \log 2, \frac{\eta}{12} \right) ,
    \end{align*}
    which implies that
    \begin{align} \label{eq:initial-conf-xieta-veta}
        \xieta{1} = \frac{1}{4} \weta{1, 1} - \log \eta = - \log 96, 
        \quad \textrm{ and } \quad
        \veta{1} = \frac{1}{2} \weta{1, 2} = \log(2).
    \end{align}
    Thus, the first claim in Eq.~\eqref{eq:xi-v-induction} is immediate at $t=1$.
    For the second part, observe that
    \begin{align*}
        \langle \weta{1}, x_{3, \eta} \rangle = 
        -2 \log \frac{\eta}{96} + \frac{96}{\eta} \left( \log \frac{\eta}{96} \right)^2 + \frac{24 (\log 2)^2}{\eta} + \frac{\eta}{24} =
        \frac{\eta}{24} + \mathcal O(\log \eta),
    \end{align*}
    as claimed.

    Now, assume that the conditions in Eq.~\eqref{eq:xi-v-induction} hold at some fixed $t$.
    We will show that they also hold at $t+1$.
    First, recall that the GD update on $f_\mathrm{bin}$ is 
    \begin{align*}
        \weta{t+1} =
        \weta{t} + \frac{\eta}{3} \sum_{i=1}^3 q(\langle \weta{t}, x_i \rangle) \, x_i =
        \weta{t} + \frac{\eta}{3} \left( \qetap{t} x_1 + \qetam{t} x_2 + q(\langle \weta{t}, x_{3, \eta} \rangle) x_{3, \eta} \right).
    \end{align*}
    Using $\qthreeeta \coloneqq q(\langle \weta{t}, x_{3, \eta} \rangle)$, this implies that 
    \begin{align} \label{eq:xi-eta-3d-update}
        \xieta{t+1} &=
        \frac{1}{4} \weta{t+1, 1} - \log \eta =
        \xieta{t} + \frac{\eta}{48} (\qetap{t} + \qetam{t}) + \frac{\eta}{12} \, \qthreeeta \left( -\frac{1}{2} + \frac{24}{\eta} \log\frac{\eta}{96} \right), \\
        \veta{t+1} &= 
        \frac{1}{2} \weta{t+1, 2} =
        \veta{t} + \frac{\eta}{12} \left( \qetap{t} - \qetam{t} \right) + \frac{12 \log(2)}{6} \, \qthreeeta .
    \end{align}

    Recall that by Eq.~\eqref{eq:margins-wrt-xi-v}, the logistic weights of the first two points satisfy
    \begin{align*}
        \eta \qetap{t} =
        \frac{1}{\eta^{-1} + e^{\xieta{t} + \veta{t}}},
        \qquad \textrm{ and } \qquad
        \eta \qetam{t} =
        \frac{1}{\eta^{-1} + e^{\xieta{t} - \veta{t}}}.
    \end{align*}
    Using the first part of the induction hypothesis, we therefore obtain
    \begin{align} \label{eq:etaq-convergence}
        \eta \qetap{t} \to e^{-\xiinf{t-1} - \vinf{t-1}},
        \qquad \textrm{ and } \qquad
        \eta \qetam{t} \to e^{-\xiinf{t-1} + \vinf{t-1}}.
    \end{align}

    Additionally, by the second part of the induction hypothesis,
    \begin{align*}
        \langle \weta{t}, x_{3,\eta} \rangle - \log \eta =
        \frac{\eta}{24} + \mathcal O(\log \eta) - \log \eta
        \to +\infty.
    \end{align*}
    Thus, Lemma~\ref{lemma:eta-scale} gives $\eta \ell(\langle \weta{t}, x_{3,\eta} \rangle) \to 0$.
    Since $q(z) \leq \ell(z)$, it follows that $\eta \qthreeeta \to 0 $.
    Using these limits together with the update rules from Eq.~\eqref{eq:xi-eta-3d-update}, we obtain
    \begin{align*}
        \xieta{t+1} \to
        \xiinf{t-1} + \frac{1}{48} \left( e^{-\xiinf{t-1} - \vinf{t-1}} + e^{-\xiinf{t-1} + \vinf{t-1}} \right) = 
        \xiinf{t-1} + \frac{1}{24} e^{-\xiinf{t-1}} \cosh(\vinf{t-1}) =
        \xiinf{t},
    \end{align*}
    where the last equality follows from Eq.~\eqref{eq:xi-v-limiting-system}, since $a = \frac{1}{4}$.
    Similarly,
    \begin{align*}
        \veta{t+1} \to
        \vinf{t-1} + \frac{1}{12} \left( e^{-\xiinf{t-1} - \vinf{t-1}} - e^{-\xiinf{t-1} + \vinf{t-1}} \right) =
        \vinf{t-1} - \frac{1}{6} e^{-\xiinf{t-1}} \sinh(\vinf{t-1}) =
        \vinf{t},
    \end{align*}
    where again in the last step we used Eq.~\eqref{eq:xi-v-limiting-system} with $b = \frac{1}{2}$.
    Therefore,
    \begin{align*}
        \lim_{\eta \to \infty} (\xieta{t+1}, \veta{t+1}) = (\xiinf{t}, \vinf{t}),
    \end{align*}
    which proves the first condition in Eq.~\eqref{eq:xi-v-induction} at time $t+1$.
    It remains to prove the second condition at time $t+1$.
    Taking the inner product of the GD update with $x_{3,\eta}$, we have
    \begin{align} \label{eq:x3-margin-update}
        \langle \weta{t+1}, x_{3,\eta} \rangle =
        \langle \weta{t}, x_{3,\eta} \rangle + \frac{\eta}{3} \Big(
            \qetap{t} \langle x_1, x_{3,\eta} \rangle
            + \qetam{t} \langle x_2, x_{3,\eta} \rangle
            + \qthreeeta \| x_{3,\eta} \|^2
        \Big).
    \end{align}
    From the definition of the dataset, $x_{3,\eta} = \mathcal O(1)$ as $\eta \to \infty$.
    Since $x_1$ and $x_2$ are fixed, it follows that
    \begin{align*}
        \langle x_1, x_{3,\eta} \rangle = \mathcal O(1),
        \qquad
        \langle x_2, x_{3,\eta} \rangle = \mathcal O(1),
        \qquad \textrm{ and } \qquad
        \|x_{3,\eta}\|^2 = \mathcal O(1).
    \end{align*}
    Since $\eta \qetap{t}$ and $\eta \qetam{t}$ converge by Eq.~\eqref{eq:etaq-convergence}, while $\eta \qthreeeta \to 0$, from Eq.~\eqref{eq:x3-margin-update}, it follows that
    \begin{align*}
        \langle \weta{t+1}, x_{3,\eta} \rangle =
        \langle \weta{t}, x_{3,\eta} \rangle + \mathcal O(1).
    \end{align*}
    Using the second part of the induction hypothesis, we conclude that
    \begin{align*}
        \langle \weta{t+1}, x_{3,\eta} \rangle =
        \frac{\eta}{24} + \mathcal O(\log \eta) + \mathcal O(1) = 
        \frac{\eta}{24} + \mathcal O(\log \eta).
    \end{align*}
    Thus, both conditions in Eq.~\eqref{eq:xi-v-induction} hold at time $t+1$, which completes the induction.
\end{proof}

The preceding result gives the required lower bound on the loss over any fixed number of iterations.
\begin{lemma}
    Let $\weta{t}$ denote the GD iterates on the dataset in Eq.~\eqref{eq:3d-dataset}, starting from $\weta{0} = 0$.
    Then, for every $N \geq 1$, there exists $\eta_N > 0$ such that for all $\eta \geq \eta_N$,
    \begin{align*}
        f_\mathrm{bin}(\weta{t}) > \frac{1}{\eta}
        \qquad \textrm{ for all } 0 \leq t \leq N.
    \end{align*}
\end{lemma}
\begin{proof}
    We first use Lemma~\ref{lemma:3d-induction} to control the rescaled loss.
    For every fixed $t \geq 1$,
    \begin{align*}
        \rhoeta{t} = \frac{\eta}{3} \left( \ell(\log \eta + \xieta{t} + \veta{t}) + \ell(\log \eta + \xieta{t} - \veta{t}) + \ell(\langle \weta{t}, x_{3, \eta} \rangle) \right).
    \end{align*}
    By the first condition in Eq.~\eqref{eq:xi-v-induction} and Lemma~\ref{lemma:eta-scale}, we have
    \begin{align*}
        \eta \ell(\log \eta + \xieta{t} + \veta{t}) \to e^{-\xiinf{t-1} - \vinf{t-1}},
        \quad\quad \textrm{ and } \quad\quad
        \eta \ell(\log \eta + \xieta{t} - \veta{t}) \to e^{-\xiinf{t-1} + \vinf{t-1}}.
    \end{align*}
    On the other hand, by the second condition in Eq.~\eqref{eq:xi-v-induction} and Lemma~\ref{lemma:eta-scale}, $ \eta \ell(\langle \weta{t}, x_{3,\eta} \rangle) \to 0$.
    Thus,
    \begin{align*}
        \rhoeta{t} \to
        \frac{1}{3} \left( e^{-\xiinf{t-1} - \vinf{t-1}} + e^{-\xiinf{t-1} + \vinf{t-1}} \right) = 
        \frac{2}{3} e^{-\xiinf{t-1}} \cosh(\vinf{t-1}) = 
        \rhoinf{t-1}.
    \end{align*}

    By Lemma~\ref{lemma:invariant-region-specialized}, we have $\rhoinf{t} \geq 72$ for all $t \geq 0$.
    Therefore, for every fixed $t \geq 1$, it holds that $\rhoeta{t} \to \rhoinf{t-1} \geq 72$.
    Fix $N \geq 1$.
    Since there are only finitely many $t \in \{1, \dots, N\}$, there exists $\eta_N > 0$ such that $\rhoeta{t} > 36$ for all $1 \leq t \leq N$ and every $\eta \geq \eta_N$.
    Then, since $\rhoeta{t} = \eta f_\mathrm{bin}(\weta{t})$, we obtain
    \begin{align*}
        f_\mathrm{bin}(\weta{t}) =
        \frac{\rhoeta{t}}{\eta} >
        \frac{36}{\eta} >
        \frac{1}{\eta}
        \qquad \textrm{ for all } 1 \leq t \leq N \textrm{ and every } \eta \geq \eta_N.
    \end{align*}
    Finally, at $t=0$, we have $f_\mathrm{bin}(\weta{0}) = \log(2) > \frac{1}{\eta}$ for all sufficiently large $\eta$.
    Thus, increasing $\eta_N$ if necessary, the conclusion holds for all $0 \leq t \leq N$.
\end{proof}

\subsection{Checking the dataset assumptions}

The norm constraints are immediate for the first two points.
For the third point,
\begin{align*}
    \| x_{3, \eta} \|^2 =
    \left( -\frac{1}{2} + \frac{24}{\eta} \log\frac{\eta}{96} \right)^2 + \left(\frac{12\log 2}{\eta}\right)^2 + \frac{1}{4}
    \to \frac{1}{2},
\end{align*}
so $\| x_{3, \eta} \| < 1$ for all sufficiently large $\eta$.
To verify linear separability with a lower bounded margin, consider the unit vector $w^* \coloneqq \frac{(2,0,3)}{\sqrt{13}} $.
Then, for all sufficiently large $\eta$,
\begin{align*}
    \langle w^*,x_1\rangle =
    \langle w^*,x_2\rangle =
    \frac{1}{2\sqrt{13}},
    \qquad
    \langle w^*,x_{3,\eta}\rangle =
    \frac{1}{\sqrt{13}} \left( \frac{1}{2} + \frac{48}{\eta}\log\frac{\eta}{96} \right),
\end{align*}
and so $\gamma(X_\eta) \geq \frac{1}{2\sqrt{13}} $.

\subsection{Proof of additional claims of Theorem~\ref{thm:binary-counterexample}}

\subsubsection{Termination of the hard perceptron~\ref{thm:hp-transition-1}}

The same family also illustrates why the dependence of the step-size threshold on $\delta$ cannot be avoided.
Recall the dataset family $X_\eta$~\eqref{eq:3d-dataset} with all labels equal to $1$, and let $\delta_\eta$ denote their corresponding (binary) gaps, as defined in Eq.~\eqref{eq:binary-gap}.
Starting from the origin, the first iterate of the binary hard perceptron~\eqref{eq:binary-perceptron} is
\begin{align*}
    z_1^\infty =
    \frac{1}{6} (x_1+x_2+x_{3,\eta}) =
    \left(
        \frac{4}{\eta}\log\frac{\eta}{96},
        \frac{2\log2}{\eta},
        \frac1{12}
    \right),
\end{align*}
and the corresponding scores are
\begin{align*}
    \langle z_1^\infty, x_1 \rangle = \frac{1}{\eta} \log\frac{\eta}{48},
    \qquad
    \langle z_1^\infty,x_2 \rangle = \frac{1}{\eta} \log\frac{\eta}{192},
    \qquad
    \langle z_1^\infty,x_{3, \eta} \rangle = \frac{1}{24} -\frac{2}{\eta} \log\frac{\eta}{96} + \mathcal O \left( \frac{\log^2\eta}{\eta^2} \right).
\end{align*}
Thus, for $\eta$ large enough, all three scores are positive and the hard perceptron dynamics terminates after one iteration.
Its corresponding gap parameter is therefore the minimum score, that is
\begin{align} \label{eq:counterexample-delta}
    \delta_\eta =
    \frac{1}{\eta} \log\frac{\eta}{192}.
\end{align}

\subsubsection{Sharpness of the $\delta$-dependency for transitioning to the stable regime~\ref{thm:delta-dep-transition}}
Now fix the dataset $X_\eta$ and run GD with a generic step size $\alpha>0$.
The first iterate is
\begin{align*}
    w_1^\alpha =
    \frac{\alpha}{6}(x_1+x_2+x_{3,\eta}) =
    \alpha z_1^\infty.
\end{align*}
Using only the contribution of the second point, we obtain
\begin{align*}
    f_\mathrm{bin}(w_1^{\alpha}) \geq
    \frac{1}{3} \ell(\langle w_1^\alpha, x_2 \rangle) = 
    \frac{1}{3} \ell(\alpha \delta_\eta ) = 
    \frac{1}{3} \log\left( 1 + e^{-\alpha \delta_\eta} \right).
\end{align*}
Since $\alpha \delta_\eta \geq 0$, we have $e^{-\alpha \delta_\eta} \in (0, 1]$.
Using the inequality $\log(1+x) \geq x/2$ for $x \in [0,1]$, we obtain a general lower bound valid for any $\alpha > 0$:
\begin{align} \label{eq:f-bin-lower-bound}
    f_\mathrm{bin}(w_1^{\alpha}) \geq
    \frac{1}{6} e^{-\alpha \delta_\eta}.
\end{align}
To prove the binary claim in Theorem~\ref{thm:binary-counterexample}, fix $c \in (0, 1)$ and choose the step size $\alpha = \frac{c}{\delta_\eta} \log \frac{1}{\delta_\eta}$.
For all sufficiently large $\eta$, substituting this $\alpha$ into Eq.~\eqref{eq:f-bin-lower-bound} gives
\begin{align*}
    f_\mathrm{bin}(w_1^{\alpha}) \geq
    \frac{1}{6} e^{-c \log(1/\delta_\eta)} = 
    \frac{1}{6} \delta_\eta^c.
\end{align*}
Since $c < 1$ and $\delta_\eta \to 0$, this implies
\begin{align*}
    \alpha f_\mathrm{bin}(w_1^{\alpha}) \geq
    \frac{c}{6} \delta_\eta^{c-1} \log\left( \frac{1}{\delta_\eta} \right)
    \xrightarrow{\eta \to \infty} \infty.
\end{align*}
Thus, $f_\mathrm{bin}(w_1^{\alpha}) > 1/\alpha$ for all sufficiently large $\eta$.
Also, the initialization satisfies $f_\mathrm{bin}(w_0^\alpha) = \log 2 > \frac{1}{\alpha}$ for all sufficiently large $\eta$, meaning GD has not entered the stable region after one iteration.
Note that since $\delta_\eta = \frac{\log(\eta/192)}{\eta}$, we have 
\begin{align*}
    \frac{\alpha}{\eta} =
    c \; \frac{\log \eta - \log\log(\eta / 192)}{\log \eta - \log 192} 
    \xrightarrow{\eta \to \infty} c,
\end{align*}
so the step size $\alpha$ is indeed of the same scale as the data parameter~$\eta$.

\subsubsection{Sharpness of the $\delta$-dependency for obtaining $\varepsilon$-loss~\ref{thm:delta-dep-eps-loss}}

Finally, for any $\varepsilon \in (0, \frac{1}{6})$, if we consider step size $\alpha < \frac{1}{\delta_\eta} \log(\frac{1}{6 \varepsilon})$, then Eq.~\eqref{eq:f-bin-lower-bound} gives
\begin{align*}
    f_\mathrm{bin}(w_1^\alpha) \geq
    \frac{1}{6} e^{-\alpha \delta_\eta} > \varepsilon,
\end{align*}
which is the claimed bound.

\end{document}

%% file: preamble.tex
\usepackage[colorlinks=true,bookmarks=true,linkcolor=blue,urlcolor=blue,citecolor=blue,breaklinks=true]{hyperref}
\usepackage{breakcites}
\usepackage{bbm}

\usepackage{doi}

\usepackage{amsmath}
\usepackage{amsfonts}
\usepackage{amsthm}
\usepackage{amssymb}

\usepackage{mathtools}

\usepackage{xcolor}
\usepackage{bm} 
\usepackage{dsfont} 

\usepackage{graphicx}
\usepackage{tikz-cd}
\usepackage{tikz}
\definecolor{mycolor1}{rgb}{0.105882,0.619608,0.466667}
\definecolor{mycolor2}{rgb}{0.85098,0.372549,0.00784314}
\definecolor{mycolor3}{rgb}{0.458824,0.439216,0.701961}
\definecolor{mycolor4}{rgb}{0.905882,0.160784,0.541176}
\definecolor{mycolor5}{rgb}{0.4,0.65098,0.117647}
\definecolor{mycolor6}{rgb}{0.65098,0.462745,0.113725}
\definecolor{mycolor7}{rgb}{0.901961,0.670588,0.00784314}
\definecolor{mycolor8}{rgb}{0.4,0.4,0.4}
\definecolor{mycolor9}{rgb}{0.301961,0,0.294118}
\definecolor{mycolor10}{rgb}{0.0313725,0.25098,0.505882}

\usepackage{psfrag}

\usepackage{algorithm}
\usepackage{algpseudocode}
\algdef{SE}[DOWHILE]{Do}{doWhile}{\algorithmicdo}[1]{\algorithmicwhile\ #1}%

\usepackage{mathrsfs} 

\newcommand{\argmax}[1]{\underset{#1}{\operatorname{argmax}}}

\newcommand{\reals}{{\mathbb{R}}}

\newcommand{\diag}{\mathrm{diag}}

\newcommand{\TODOF}[1]{}

\newtheorem{theorem}{Theorem}[section]
\newtheorem{lemma}[theorem]{Lemma}
\newtheorem{proposition}[theorem]{Proposition}
\newtheorem{corollary}[theorem]{Corollary}
\newtheorem{remark}[theorem]{Remark}

\usepackage[lmargin=3cm,rmargin=3cm,bottom=3cm,top=3cm]{geometry}
\usepackage[utf8]{inputenc}
\usepackage[round,compress]{natbib}

\usepackage{sectsty}
\makeatletter\def\@seccntformat#1{\protect\makebox[0pt][r]{\csname the#1\endcsname\hspace{12pt}}}\makeatother

\usepackage{subcaption}

\usepackage{enumitem} 

\usepackage[normalem]{ulem} 

\usepackage{multicol}

\usepackage{blindtext} 

\newcommand{\aref}[1]{\hyperref[#1]{A\ref{#1}}}

\newcommand{\seta}[1]{s^{(#1)}}
\newcommand{\sinf}[1]{\tilde s^{(#1)}}

\newcommand{\ZZ}[1]{Z^{(#1)}}
\newcommand{\Seta}[1]{s^{(#1)}}
\newcommand{\Sinf}[1]{\tilde s^{(#1)}}
\newcommand{\ZZeta}[1]{Z^{(#1)}}
\newcommand{\Zzeta}[1]{z^{(#1)}}
\newcommand{\WWeta}[1]{W^{(#1)}}
\newcommand{\Wweta}[1]{w^{(#1)}}
\newcommand{\ZZinf}[1]{\tilde Z^{(#1)}}
\newcommand{\Zzinf}[1]{\tilde z^{(#1)}}

\newcommand{\nueta}[1]{\nu^{(#1)}}
\newcommand{\nuinf}[1]{\tilde \nu^{(#1)}}

\newcommand{\Zomega}[1]{\omega^{(#1)}}

\newcommand{\pp}[1]{p^{(#1)}}
\newcommand{\qq}[1]{q^{(#1)}}

\newcommand{\SSS}[1]{S^{(#1)}}
\newcommand{\SSSinf}[1]{\tilde S^{(#1)}}

\newcommand{\Ent}{\mathsf{Ent}}
\newcommand{\winnernoarg}{\widetilde{\mathcal W}}
\newcommand{\winner}[1]{\widetilde{\mathcal W}^{(#1)}}

\newcommand{\softmax}{\mathrm{softmax}}
\newcommand{\hardmax}{\mathrm{hardmax}}

\newcommand{\xp}{x_+}
\newcommand{\xm}{x_-}
\newcommand{\weta}[1]{w_{#1}^{(\eta)}}

\newcommand{\veta}[1]{v_{#1}^{(\eta)}}
\newcommand{\xieta}[1]{\xi_{#1}^{(\eta)}}
\newcommand{\qetap}[1]{q_{+,#1}^{(\eta)}}
\newcommand{\qetam}[1]{q_{-,#1}^{(\eta)}}
\newcommand{\qthreeeta}{q_3^{(\eta)}}
\newcommand{\rhoeta}[1]{\rho_{#1}^{(\eta)}}

\newcommand{\meta}{m^{(\eta)}}
\newcommand{\reta}{r^{(\eta)}}

\newcommand{\vinf}[1]{\tilde v_{#1}}
\newcommand{\xiinf}[1]{\tilde \xi_{#1}}
\newcommand{\rhoinf}[1]{\tilde \rho_{#1}}

\newcommand{\wbin}{v}